\documentclass[12pt]{article}
\usepackage{amsmath,amsfonts,amssymb}
\usepackage{alltt}

\usepackage{amsmath,amsfonts,ifthen,fullpage,enumerate}  
\usepackage{mathrsfs}
\usepackage{hyperref}
\usepackage{mathtools}
\usepackage{array}
\usepackage{tikz-cd}
\usetikzlibrary{arrows.meta}
\tikzset{>={Latex[width=2mm,length=2mm]}}

\usepackage{colortbl}

\usepackage{enumerate}

\usepackage{amsmath,amsfonts,amssymb}

\usepackage{makecell}  

\usepackage{graphicx}        

\usepackage{tikz-cd}

\usepackage{tcolorbox}
\def\greyboxvar#1#2
{\begin{tcolorbox} [boxrule=-5pt, boxsep=-0.1cm, colback=white!95!black, 
	width=#2] #1\end{tcolorbox}}

\def\greybox#1
{
	\begingroup
\setlength{\tabcolsep}{2pt} 
\renewcommand{\arraystretch}{1.3} 
	\begin{tabular}{c}
\cellcolor{gray!25}#1\end{tabular}
\endgroup
}

\usepackage[all,pdf]{xy}
\usepackage{lmodern,amssymb}   

\def\spam{\mathop{\rm span}\nolimits}
\def\gcd{\mathop{\rm gcd}\nolimits}
\def\lcm{\mathop{\rm lcm}\nolimits}
\def\Orb{\mathop{\rm Orb}\nolimits}

\newcommand\divides{\mid}

\def\rank{\mathop{\rm rank}\nolimits}

\def\pmat#1{\begin{pmatrix}#1\end{pmatrix}}

\def\question#1{{\bf Question: }#1}
\def\question#1{}

\def\cG{{\cal G}}

\def\cL{{\cal L}}
\def\cO{{\cal O}}
\def\cT{{\cal T}}

\def\CC{\mathbb{C}}

\def\HH{\mathbb{H}}

\def\ZZ{\mathbb{Z}}

\newtheorem{theorem}{Theorem}[section]
\newtheorem{corollary}{Corollary}[section]
\newtheorem{lemma}{Lemma}[section]
\newtheorem{example}{Example}[section]

\newtheorem{proposition}{Proposition}[section]
\newtheorem{definition}{Definition}[section]

\newenvironment{proof}{{\noindent \it
Proof.}}{\hfill$\Box$\medskip}
\newif\ifdraft\def\draft{\drafttrue\hoffset=.8truecm\showlabeltrue
\def\comment##1{{\bf comment: ##1}}
\headline={\sevenrm \hfill \ifx\filenamed\undefined\jobname\else\filenamed\fi%
(.tex) (as of \ifx\updated\undefined???\else\updated\fi)
 \TeX'ed at {\hour\time\divide\hour by 60{}%
\minutes\hour\multiply\minutes by 60{}%
\advance\time by -\minutes
\the\hour:\ifnum\time<10{}0\fi\the\time\  on \today\hfill}}
}

\def\inpro#1{\langle#1\rangle}
\def\ip#1{\langle\kern-.28em\langle#1\rangle\kern-.28em\rangle_\nu}

\def\cD{{\cal D}}
\def\cO{{\cal O}}
\def\cT{{\cal T}}
\def\cI{{\cal I}}
\def\cH{{\cal H}}

\def\openR{{{\rm I}\kern-.16em {\rm R}}}

\let\ga\alpha

\let\gb\beta

\let\gga\gamma

\let\gep\varepsilon

\let\gL\Lambda

\let\gs\sigma

\let\go\omega
\let\gO\Omega
\let\ga\alpha
\let\gga\gamma

\let\gb\beta

\let\gs\sigma
\def\inpro#1{\langle#1\rangle}

\def\Iff{\hskip1em\Longleftrightarrow\hskip1em}
\def\Implies{\hskip1em\Longrightarrow\hskip1em}

\def\mod{\rm\ mod\ }

\font\bfmedtype=cmbx10 scaled\magstep1 
\def\formeq{\the\sectionno.\the\equationno}  
\def\elabel#1/#2/#3/{\global\advance\equationno by 1 %
\ifx#1\empty\else\emember#1%
\ifshowlabel\marginal{\string#1}\fi\fi%
\ifmmode\eqno{#3(\formeq#2)}\else#3\formeq#2\fi} 

\def\makeblanksquare#1#2{
\dimen0=#1pt\advance\dimen0 by -#2pt
      \vrule height#1pt width#2pt depth0pt\kern-#2pt
      \vrule height#1pt width#1pt depth-\dimen0 \kern-#1pt
      \vrule height#2pt width#1pt depth0pt \kern-#2pt
      \vrule height#1pt width#2pt depth0pt
}

\font\bfmedtype=cmbx10 scaled\magstep1

\title{\bf 
An elementary classification of the quaternionic reflection groups of rank two
}

\author{Shayne Waldron\\ 
 \\
Department of Mathematics \\ University of Auckland\\
Private
Bag 92019, Auckland, New Zealand\\
e--mail: waldron@math.auckland.ac.nz}

\begin{document}

\maketitle 

\begin{abstract}

We give an elementary classification and 
presentation of the finite quaternionic reflection 
groups of rank two, based on the notion of a ``reflection system''. 
This simplifies the existing classification, which is shown to be
incomplete, e.g., there exist four imprimitive quaternionic reflection groups of order $192$ with $22$ 
reflections which are not isomorphic (one of which was previously unknown).


\end{abstract}

\bigskip
\vfill


 \ifdraft (Other designs, configurations) \else\fi
15B33, \ifdraft (Matrices over special rings (quaternions, finite fields, etc.) \else\fi
20C25, \ifdraft Projective representations and multipliers \else\fi
20G20, \ifdraft Linear algebraic groups over the reals, the complexes, the quaternions \else\fi
51M05, \ifdraft (Euclidean geometries (general) and generalizations) \else\fi
51M20, \ifdraft (Polyhedra and polytopes; regular figures, division of spaces [See also 51F15]) \else\fi
\quad
secondary
15B57, \ifdraft (Hermitian, skew-Hermitian, and related matrices) \else\fi
51E99, \ifdraft Geometry,  None of the above, but in this section \else\fi
51M15, \ifdraft (Geometric constructions in real or complex geometry) \else\fi
65D30. \ifdraft (Numerical integration) \else\fi

\vskip .5 truecm
\hrule
\newpage

\section{Introduction}

The finite (unitary) quaternionic reflection groups were formally introduced and  then
classified by Cohen \cite{C80}, with various results then built upon it, e.g., 
the classification of parabolic subgroups \cite{BST23}, \cite{S23}, and certain conformal field theories
\cite{DZ24}. 

This followed in the spirit of the real reflection groups
which were classified by Coxeter \cite{C34}, and the complex reflection groups 
which were classified by Shephard and Todd \cite{ST54}. 
There are various relationships between these groups,
e.g., Cohen understood that the quaternionic reflection groups in one dimension were
classified by Stringham \cite{St1881}, and recognised that some imprimitive 
collineation groups for $\CC^4$ of Blichfeldt \cite{B17} give rise to primitive quaternionic reflection groups of rank $2$
(see \cite{K81} for an interesting overview, and \cite{Z97}, \cite{CS03}, \cite{V21} for general background). 

Here we consider the classification of the imprimitive quaternionic reflection groups of
rank two. 
These are finite groups of unitary $2\times 2$ monomial matrices over the quaternions $\HH$, 
which are generated by reflections, i.e., nonidentity matrices $g$ 
which fix a $1$-dimensional subspace (this is equivalent to $\rank(g-I)=1$).
Elementary calculations show that there are two types of reflections
$$ \pmat{h&0\cr0&1}, \ \pmat{1&0\cr0&h}, \quad h\ne1, \qquad
\pmat{0&b\cr b^{-1}&0},  $$
which have orders the order of $h$ and $2$, respectively. The roots for these 
reflections, i.e., a vector in the orthogonal complement of the fixed subspace
are $e_1$, $e_2$, $(1,-\overline{b})$, respectively.
The six primitive quaternionic reflection groups of rank two can be obtained
by adding single non-monomial reflection to an 
imprimitive reflection group \cite{W24}, \cite{BW25}.

Let $G$ be an irreducible imprimitive quaternionic reflection group
of rank two.
Then $G$ is conjugate in $U(\HH^2)$ to a unitary group where the reflections are
\begin{equation}
\label{Greflections}
\pmat{h&0\cr0&1}, \ \pmat{1&0\cr0&h}, \quad h\in H,\ h\ne1, \qquad
\pmat{0&b\cr b^{-1}&0}, \quad b\in L,
\end{equation}
where $1\in L$ (a standardisation condition), $1\in H$, and
$H$ and $L$ are subsets of $U(\HH)$.

\begin{itemize}
\item Since the reflections for a given root (e.g., $e_1$)
are a subgroup of $G$, we conclude that $H$ is a group.
\item Since the set of reflections in $G$ is closed under conjugation, we have that
\begin{align*}
&\pmat{h&0\cr0&1}\pmat{0&1\cr 1&0}\pmat{h&0\cr0&1}^{-1}
= \pmat{0&h\cr h^{-1}&0}\in G, \cr
&\pmat{0&b\cr b^{-1}&0}\pmat{h &0\cr0& 1}\pmat{0&b\cr b^{-1}&0}^{-1}
= \pmat{1&0\cr 0& b^{-1}h b}\in G, \qquad h\in H, \ b\in L.
\end{align*}
Hence $H\subset L$, and $H$ is a normal subgroup
of the group $K=\inpro{L}$ generated by $L$.
\end{itemize}
Therefore, each group $G$ of the above form is determined by $H\lhd K$, which are
uniquely defined up to conjugation in $U(\HH)$, the reflections in $G$,
and a suitable $L$.

Clearly, not every $H\lhd K$ and $L\subset K$ gives a reflection group
as above, and so the classification of the imprimitive reflection groups of rank two
can be reduced to finding the suitable subsets $L$ (and $H$),
and then any isomorphisms between the reflection groups that they give
-- just as was done by Shephard and Todd, and in turn Cohen. 
Our contribution to this approach is the
observation that the $L$ further satisfies certain algebraic properties which 
we describe as being a ``reflection system''.
The remainder of the paper is devoted to determining all the reflection systems, 
and consequently classifying the imprimitive quaternionic reflection groups
of rank two (Theorem \ref{CombinedClassification}).
The existing classification is shown to be incomplete, and with some double counting.

\section{Reflection systems and the canonical form}

Let $G$ be the imprimitive reflection group
generated by the reflections of (\ref{Greflections}). Then
\begin{align*}
& \pmat{0&b\cr b^{-1}&0} \pmat{1&0\cr0&h}
=\pmat{0& bh\cr (bh)^{-1}}\in G, \qquad h\in H, \ b\in L, \cr
& \pmat{0&a\cr a^{-1}&0}\pmat{0&b\cr b^{-1}&0}\pmat{0&a\cr a^{-1}&0}^{-1}
=\pmat{0&ab^{-1}a\cr (ab^{-1}a)^{-1}&0}\in G, \qquad a,b\in L,
\end{align*}
so that
\begin{enumerate}
\item $LH=L$, i.e., $L$ is a union of cosets in the group $K/H$ (and so $|H|$ divides $|L|$).
\item $ab^{-1}a\in L$, $\forall a,b\in L$, i.e.,
$L$ is closed under the binary operation $(a,b)\mapsto ab^{-1}a$.
\end{enumerate}
Hence, there is an irreducible imprimitive quaternionic reflection group
generated by reflections of the form (\ref{Greflections}) if and only if $H\lhd K$ and
$H\subset L \subset K $ satisfies $K=\inpro{L}$ and the two conditions
above hold, and we denote this group by $G=G(K,L,H)$. 

Let $L_G$ give the set of nondiagonal reflections in $G=G(K,L,H)$, i.e.,
\begin{equation}
\label{LGform} 
L_G := \{ b\in K: \pmat{0&b\cr b^{-1}&0} \in G\}\supset L.
\end{equation}
If all the nondiagonal reflections in $G=G(K,L,H)$ are given by $L$,
i.e., $L=L_G$, then we say that $G$ is in {\bf canonical form},
as a subgroup of $G(K,K,K)$, and we denote it by
\begin{equation}
\label{canonicalform}
G_K(L,H)=G(L,H)=G(K,L,H),
\end{equation}
and call this unique $(K,L,H)$, or $(L,H)$,
the {\bf canonical label} for the group $G$,
as a subgroup of $G_K(K,K)$.
We observe that for the canonical label
\begin{equation}
\label{HLdefn}
H_L := \bigl\{ h : \pmat{h&0\cr0&1}\in\inpro{\bigl\{\pmat{0&b\cr b^{-1}&0}\bigr\}_{b\in L}}\bigr\} \subset H.
\end{equation}

Clearly,
\begin{equation}
\label{GKLinclusions1}
G(K,L_1,H_1)\subset G(K,L_2,H_2), \qquad L_1\subset L_2,\ H_1\subset H_2,
\end{equation}
for suitable $(L_j,H_j)$, including the canonical labels. Indeed,
\begin{equation}
\label{GKLHsinclusion}
G(K_1,L_1,H_1)\subset G(K_2,L_2,H_2), \qquad
K_1\subset K_2,\ L_1\subset L_2,\ H_1\subset H_2,
\end{equation}
and so it may be possible to obtain the same reflection group for $K_1\ne K_2$.
This happens in two instances (see Example \ref{isomorphismI} and Theorem \ref{abapbp-lemma}).

We now define the fundamental algebraic object of our investigation.

\begin{definition}
We call a subset $L$ of $K$ (a finite group) a {\bf reflection system}
for $K$ if
\begin{enumerate}
\item $K=\inpro{L}$.
\item $L$ is closed under the binary operation $(a,b)\mapsto a \circ b := ab^{-1}a$.
\item $1\in L$.
\end{enumerate}
\end{definition}

We say that reflection systems $L$ and $L'$ for $K$ are {\bf isomorphic} if there
is a bijection $\psi:L\to L'$ which  
preserves $\circ$, i.e.,
$\psi(a\circ b) = \psi(a) \circ \psi(b)$.

\begin{example} The group $K=\ZZ_2\times\ZZ_2$ has two reflection systems
$$ L_3=\{(0,0),(1,0),(0,1)\},\qquad
L_4=\ZZ_2\times\ZZ_2=\{(0,0),(1,0),(0,1),(1,1)\}. $$
\end{example}

We will often define a reflection system $L$ via a {\bf generating} set, i.e., 
a subset $X$ of $L$ whose closure under the binary operation $\circ$
is $L$, and we write $L=L(X)$. We note that
\begin{itemize}
\item For any reflection group $G=G(K,L,H)$, $L$ is a reflection system for $K$.
\item $L=K$ is a reflection system for $K$, and
this is the only case where $L$ is a group.
\item Since $1\circ x=x^{-1}$, reflection systems are closed under taking inverses.
\item Since $x^{-1}\circ 1=x^2$, reflection systems are closed under taking squares.
\end{itemize}
Moreover, if $L$ is a reflection system and $x\in L$ (so that $x^{-1}\in L$), then 
it follows that
$xL$ and $Lx$ are reflection systems, via the calculations
\begin{align*}
&\hbox{$xa \circ xb = xa(xb)^{-1}xa=xab^{-1}x^{-1}xa= xab^{-1}a=x(a \circ b)$}, \cr
&\hbox{$ax \circ bx = ax(bx)^{-1}ax=axx^{-1}b^{-1}ax= ab^{-1}ax=(a \circ b)x$}, \qquad
	a,b\in L.
\end{align*}
The corresponding reflections for these systems are conjugate in $U(\HH^2)$ to those for $L$,
via the calculations
$$ 
\pmat{x&0\cr 0&1}^{-1}M_{xb}\pmat{x&0\cr 0&1}=M_b, \qquad
\pmat{1&0\cr 0&x}M_{bx}\pmat{1&0\cr 0&x}^{-1}=M_b, 
\qquad M_b:=\pmat{0&b\cr b^{-1}&0}. $$
Therefore, we consider reflection systems $L$ and $L'$ for $K$ to be {\bf equivalent} if
$L'$ equals $xL$ or $Lx$, for some $x\in L$, up to an automorphism of $K$.
This is an equivalence relation, and equivalent reflection systems lead to the same
reflection groups $G(K,L,H)$.
Clearly, equivalent reflection systems are isomorphic.

\begin{lemma} 
\label{Ageneratelemma}
Suppose that $\{1\}\cup A$ generates a reflection system $L$ for $K$.
If $x\in L$, then $\{1,x\}\cup xA$ and $\{1,x\}\cup Ax$ generate the reflection
systems $xL$ and $Lx$ for $K$, which are equivalent to (but possibly not equal to) $L$.
\end{lemma}

\begin{proof} We have observed that $xL$ and $Lx$ are closed under $\circ$,
for any $x$, and that they are generated by $x(\{1\}\cup A)$ and $(\{1\}\cup A)x$.
If $x\in L$, then $x^{-1}=x\circ 1\in L$,
so that $1=x^{-1}\in xL$ and $1=x^{-1}x\in Lx$, and hence  $\{1,x\}\cup xA$ and $\{1,x\}\cup Ax$ generate $xL$ and $Lx$.
The groups generated by the above two sets contains
$A=x^{-1} (xA)=(Ax)x^{-1}$, and so they generate $K=\inpro{A}$.
\end{proof}

The binary polyhedral groups are considered 
in Section \ref{binarypolysect}. 
It is convenient to now consider the following particular example for the purposes
of illustration and motivation.

\begin{example}
\label{TrefsysExample}
For the binary tetrahedral group of order $24$
\begin{equation}
\cT :=\inpro{i,{1+i+j+k\over 2}}, 
\end{equation}
the reflection system $\cT$ can be generated as follows
\begin{equation}
\label{L24T}
	L_{24}^\cT := \cT = L(\{1,i,j,{1+i+j+k\over 2}\}).
\end{equation}
There is also a second reflection system of size $12$ given by
\begin{equation}
\label{L12T}
L_{12}^\cT := L(\{1,i,{1+i+j+k\over 2}\}).
\end{equation}
There are twelve equivalent copies of this reflection system in $\cT$. 
Six of these can be obtained by left or right multiplication, namely
$L_{12}^\cT$, $iL_{12}^\cT$, $jL_{12}^\cT$, $kL_{12}^\cT$, and
	$$ \zeta^{-1}L_{12}^\cT=L(\{1,j,\zeta\}), \quad
	\zeta L_{12}^\cT= L(\{1,k,\zeta\}), \qquad \zeta:={1+i+j+k\over 2}. $$
The automorphism $i\mapsto j$, $j\mapsto i$, gives another copy
$$ L(\{1,j,{1+i+j+k\over 2}\}), $$
and multiplication of this (as above) gives the remaining six copies.
\end{example}

Reflection systems are constructed implicitly in \cite{C80}, via
$$ L=L_\ga:=\{x\in K: \ga(xH)=x^{-1}H\}, $$
where $\ga$ is an automorphism of $K/H$ ($H\lhd K$)
of order one (the identity) or two.
This is easily verified, e.g., for $a,b\in L_\ga$, we have
\begin{align*}
\ga(ab^{-1}aH)
&=\ga(aH)\ga(b^{-1}H)\ga(aH) \cr
&=(a^{-1}H)(bH)(a^{-1}H)
= a^{-1}ba^{-1}H
= (ab^{-1}a)^{-1}H,
\end{align*}
so that $ab^{-1}a\in L_\ga$.
Conversely, if $L$ is a reflection system for $K$, and $H\lhd K$ with $L=LH$, then
\begin{equation}
\label{ga-auto-defn}
\ga(xH):=x^{-1}H, \qquad x\in L,
\end{equation}
defines an automorphism of $K/H$ of order $\le 2$.

For the groups of order $96$ and $48$ constructed in \cite{C80} (Table I) 
from what is essentially the $12$ element reflection system $L_{12}^\cT$ for $\cT$ of (\ref{L12T}),
automorphisms of order two are taken, namely
conjugation by the permutation $(1\,2)$, where $H=C_2$, $K/H\cong A_4$,
and conjugation by $i-j$ (for which $(i-j)^2=-2$), where $H=C_1$.

We can now give the structural form of a reflection group in the canonical form.

\begin{lemma}
\label{GKLHstructure}
A reflection group in the canonical form has the following elements
\begin{equation}
\label{GKLHelements}
G=G_K(L,H)=\bigl\{\pmat{b&0\cr0&b_\ga\, h}\pmat{0&1\cr1&0}^m :b\in K,\ h\in H,\ m=0,1\Bigr\},
\end{equation}
	where $b_\ga$ is any element of the coset $\ga(bH)$ of (\ref{ga-auto-defn}),
	e.g., $b_\ga=b^{-1}$, $b\in L$.
	Therefore 
\begin{enumerate}[(i)]
\item $G$ has order  $2|H||K|$.
\item $G$ has $2|H|+|L|-2$ reflections.
\end{enumerate}
\end{lemma}

\begin{proof} Since 
	$$ \pmat{a&0\cr0&b}\pmat{0&1\cr1&0} = \pmat{0&a\cr b&0}, $$
there is a $1$-$1$ correspondence between the diagonal and nondiagonal elements
of $G$. 
Hence, it suffices to determine the subgroup of diagonal matrices in $G$.
A product of two reflections given by elements of $L$ is diagonal, i.e., 
\begin{equation}
\label{productoftwo}
\pmat{0&a\cr a^{-1}&0}\pmat{0&b^{-1}\cr b&0}
=\pmat{ab&0\cr0&a^{-1}b^{-1}}, \qquad a,b\in L,
\end{equation}
so that $G$ contains the diagonal matrices of the form
\begin{equation}
\label{bproducts} \pmat{b_1b_2\cdots b_r & 0\cr 0 & b_1^{-1}b_2^{-1}\cdots b_r^{-1} h}, \qquad
b_1,\ldots,b_r \in L, \quad h\in H. 
\end{equation}
To see this, take products of the diagonal matrices in (\ref{productoftwo}), choosing
some $b=1$ to get the case when $r$ is odd, and right multiply by the appropriate
diagonal reflection.

Since $L$ generates $K$, for any $b\in K$, 
there is at least one choice in (\ref{bproducts}), with
$$ b = b_1b_2\cdots b_r. $$
	We claim that $b_1^{-1}b_2^{-1}\cdots b_r^{-1} h\in\ga(bH)$,
	so that $b_1^{-1}b_2^{-1}\cdots b_r^{-1} h =b_\ga\, h'$, with $h'\in H$.
This follows from the calculation
\begin{align*}
\ga(b_1^{-1}b_2^{-1}\cdots b_r^{-1} H)
&=\ga(b_1^{-1}H)\ga(b_2^{-1} H)\cdots \ga(b_r^{-1} H) \cr
&= (b_1 H)(b_2 H)\cdots(b_r H)
= b_1b_2\cdots b_r H =bH,
\end{align*}
	$$ \Implies b_1^{-1}b_2^{-1}\cdots b_r^{-1} H = \ga^2(b_1^{-1}b_2^{-1}\cdots b_r^{-1} H)
	=\ga(bH). $$
Since $H$ is normal in $K$, the matrices in (\ref{bproducts}) form a group,
i.e., they give all the diagonal matrices in $G$, and we obtain (\ref{GKLHelements}) 
	and hence $|G|=2|H||K|$.

Finally, from (\ref{Greflections}), $G$ has $2(|H|-1)+|L|=2|H|+|L|-2$ reflections.
\end{proof}

We can use Lemma \ref{GKLHstructure} to address the question of when two 
reflection groups in the canonical form $G_K(L,H)$ and $G_{K'}(L',H')$ give 
the same reflection group. 
For this to happen they must have the same order and number of reflections, i.e.,
\begin{equation}
\label{GKHLequivcnd}
2|H||K|=2|H'||K'|, \qquad 2|H|+|L|-2= 2|H'|+|L'|-2.
\end{equation}
There are three cases:

\begin{enumerate}[(i)]
\item If the groups $K$ and $K'$ are the same, so that $|K|=|K'|$, then this gives
$$ |H|=|H'|, \qquad |L|=|L'|. $$
Hence, the groups $G_K(L,H)$ would all be different if $K$ has only one reflection system of
any given size (up to equivalence) and one normal subgroup $H$ of any given order. 
It turns out (from our calculations) that this is always the case.

\item If $K$ and $K'$ are different, with the same orders, then again $|L|=|L'|$. 
But $L$ and $L'$ have different algebraic structures (they generate
different groups), and so the reflection groups generated by
the reflections that they give are not isomorphic, and correspondingly 
the groups $G_K(L,H)$ and $G_{K'}(L',H')$ are different.

\item If $K$ and $K'$ are different, with different orders, say $|K|<|K'|$, then the
	reflections in the corresponding groups must all be of order $2$, so that (see 
		Lemma \ref{isomorphismlemma})
$$ H=C_2,\quad H'=1, \qquad |K'|=2|K|, \qquad |L'|=|L|+2. $$
We will see there are cases where this can happen, and the reflection groups
are the same (Example \ref{isomorphismI} and Theorem \ref{abapbp-lemma}).
\end{enumerate}

We now outline the calculations to follow.

\section{Finite subgroups of $U(\HH)$ and the general method}

The finite subgroups $K$ of $U(\HH)$, equivalently of the unit quaternions $\HH^*$, 
were classified (up to conjugation) by 
Stringham \cite{St1881}. They are (see \cite{LT09} Theorem 5.12)
\begin{enumerate}[\it(i)]
\item The cyclic group $C_n$ of order $n$, $n\ge1$.
\item The dicyclic (or binary dihedral group) $\cD_n={\rm Dic}_n$ of order $4n$, $n\ge2$.
\item The binary tetrahedral group $\cT$ of order $24$.
\item The binary octahedral group $\cO$ of order $48$.
\item The binary icosahedral group $\cI$ of order $120$.
\end{enumerate}
The notation $\cD_n$ for the dicyclic group 
(as for the dihedral group) is not
standardised, e.g., $Q_{4n}$ is also used. 
We use the indexing of \cite{C80} for the purpose of easy comparison.
From these, we proceed as follows.

\vskip0.6truecm
\noindent {\bfmedtype General method} (Classification of imprimitive reflection
groups of rank two)
\vskip0.2truecm

\noindent
For each finite group $K\subset U(\HH)$, as above, determine the
reflection systems $L$ (there is always $L=K$). Then
for a given reflection system $L$ of $K$, construct the corresponding
reflection groups, for which there are exactly $|L|$ nondiagonal reflections 
(given by $L$), as follows:

\begin{enumerate}
\item Determine the normal subgroup $H_L\lhd K$ of (\ref{HLdefn}), to obtain
	the reflection group
\begin{equation} 
\label{Basegroup}
G_K(L,H_L),
\end{equation}
which we call the {\bf base group} for $L$.
\item Determine the other reflection groups with nondiagonal reflections given by $L$,
which we call the {\bf higher order groups} for $L$.
These contain $G_K(L,H_L)$ as a proper subgroup, and are given by
\begin{equation}
\label{Lgroups}	
G_K(L,H), \qquad H\ne H_L \quad L_H=L,
\end{equation}
where $H\lhd K$, $H_L\subset H\subset L$, with $LH=L$.
In particular, we have $|H|$ divides $|L|$.

We observe $L_{H}\subset L_{H'}$, $H\subset H'$, so that once 
an $H$ with $L_H\ne L$ has been found, none of the groups $G(K,L,H')$, 
$H'\supset H$ are in the canonical form, and hence are not 
included in our classification.
\item Determine any isomorphisms between $G_K(L,H)$ and $G_{K'}(L',H')$,
so that the list of imprimitive reflection groups has no duplicates.
\end{enumerate}
In summary:

\begin{itemize}
\item
We work up from the base group for $L$, adding diagonal reflections given by $H$,
until the point where any enlargement of $H$ introduces new nondiagonal reflections.
\end{itemize}


For any reflection group $G(K,L,H)$, it is often convenient to give a generating
set of reflections given by subsets $\cL\subset L$, $\cH\subset H$, i.e.,
$$ \pmat{h&0\cr0&1}, \quad h\in\cH, \qquad \pmat{0&b\cr b^{-1}&0}, \quad b\in\cL, $$
and we write
\begin{equation}
\label{cGdefn}
G(K,L,H)=\cG(\cL,\cH).
\end{equation}
If $b=1\in\cL$, then it suffices to take just one of the reflections corresponding to a given $h\in\cH$, since
$$ \pmat{0&b\cr b^{-1}&0}\pmat{h&0\cr0&1}\pmat{0&b\cr b^{-1}&0}
=\pmat{1&0\cr0&b^{-1}hb}. $$
Typically, one would take $\cL$ to be a generating set for $L$, and $\cH$ to be 
elements of $H\setminus H_L$ that together with $H_L$ generate $H$.
As larger $\cH$ are taken, it is often possible to remove elements from $\cL$
to get a smaller generating set. 

As an illustration, we now consider the case (i), i.e., when $K$ is the cyclic group $C_n$, 
which gives the complex reflection groups.

\begin{example} (Cyclic group)
	Let $K=C_n=\inpro{\go}$, where $\go=e^{2\pi i\over n}$ is a primitive $n$-th root of unity.
There is just one reflection system
$$ L_n=L(\{1,\go\}) = \{1,\go,\ldots,\go^{n-1}\}, $$
and these are the only complex reflection systems.
The base group
	$$ G_{C_n}(L_n,1)=\cG(\{1,\go\},\{\}), \qquad n\ge3, $$
is the dihedral group of order $2n$. 
The normal subgroups of $K=C_n$ are 
$$ H=C_{n/p}=\inpro{\{\go^p\}}, \qquad p\divides n. $$
Each of these gives a higher order group for $L_n$ (for $p\ne n$),
and we obtain the following imprimitive reflection groups for the reflection system $L_n=C_n$
$$ G_{C_n}(C_n,C_{n/p}) =\cG(\{1,\go\},\{\go^p\}) = G(n,p,2), \qquad
|G|= 2 n {n\over p}, $$
where $G(n,p,2)$ is the notation used by Shephard and Todd. 
	The inclusions between them follow immediately from (\ref{GKLinclusions1}).

	For $n=2$, we have the only real imprimitive reflection group $G(2,1,2)$
	(Heisenberg group), with the base group $G(2,2,2)$ being reducible.
\end{example}

\begin{lemma}
\label{isomorphismlemma}
Let $|K|<|K'|$. If the reflection groups $G_K(L,H)$ and $G_{K'}(L',H')$
are isomorphic,
then they can only have reflections of order $2$, so that $H=C_2$, $H'=1$, and
$$ |K'|=2|K|, \qquad |L'|=|L|+2.  $$
Moreover, $L$ is a subreflection system of $L'$, and $K$ is a subgroup of $K'$.
\end{lemma}

\begin{proof}
For any reflection group $G_K(L,H)$ the nondiagonal reflections (given by $L$) have
order $2$, and the diagonal reflections given by $h\in H$ have order the order of $h$,
and so in our case, we must have $|H|,|H'|\le2$. From the orders of the reflection
groups being equal, we then have
$$ 2|H||K|=2|H'||K'| \Implies {|K'|\over|K|}={|H|\over|H'|}=2
	\Implies H=C_2,\ H'=1, \quad |K'|=2|K|. $$
From the number of reflections in the groups being equal, we then have
$$ 2|H|+|L|-2=2|H'|+|L'|-2 \Implies 
	|L'|=|L|+2. $$
An isomorphism $G_K(L,H)\to G_{K'}(L',H')$ must map the reflections given by $L$ to
the reflections given by $L'$, say
	$$ \pmat{0&b\cr b^{-1}&0}\mapsto \pmat{0&b'\cr (b')^{-1}&0}, \quad b\in L. $$
Since the injective map $L\to L':b\mapsto b'$ preserves $\circ$, $L$ 
	is a reflection subsystem of $L'$, and $K=\inpro{L}$ is a subgroup of $K'=\inpro{L'}$.
\end{proof}

\section{The binary polyhedral groups}
\label{binarypolysect}

We now consider the binary polyhedral groups
\begin{align*}
	\cT &:=\inpro{i,{1+i+j+k\over2}} \quad\hbox{(binary tetrahedral group of order $24$)}, \cr
\cO &:=\inpro{{1+i\over\sqrt{2}},{1+i+j+k\over2}} \quad\hbox{(binary octahedral group of order $48$)}, \cr
	\cI &:=\inpro{i,{1+i+j+k\over2},{1+\tau i+\gs j\over2}},
	\quad\hbox{(binary icosahedral group of order $120$)},
\end{align*}
where 
$$ \tau:={1+\sqrt{5}\over2}, \qquad \displaystyle\gs:={1-\sqrt{5}\over2}. $$
Some insight into these groups is obtained by listing their elements and orders (see the appendix).
Each of these contains the quaternion group 
$$ Q_8=\inpro{i,j}=\{1,-1,i,-i,j,-j,k,-k\}. $$
Since $({1+i\over\sqrt{2}})^2=i$, we have the inclusions
$$ \cT \subset \cO, \qquad \cT\subset\cI. $$

We consider the first example in some detail for the purpose of
illustration.

\begin{example} (Binary tetrahedral group) 
Let $K=\cT$ be the binary tetrahedral group, 
	which
has normal subgroups $H=1,C_2,Q_8,\cT$, and two reflection systems 
(Example \ref{TrefsysExample})
$$ L_{24}^\cT := \cT = L(\{1,i,j,{1+i+j+k\over 2}\}), \qquad
L_{12}^\cT := L(\{1,i,{1+i+j+k\over 2}\}). $$
For the reflection system $L=K=L_{24}^\cT=\cT$,
we have $H_L=Q_8$, and hence
obtain the following reflection groups $G$ in the canonical form
\begin{align*}
G_\cT(\cT,\cT), & \qquad |G|=1152,\quad \hbox{$70$ reflections} \quad
\hbox{(higher order group)},  \cr
G_\cT(\cT,Q_8), & \qquad  |G|=384, \quad \hbox{$38$ reflections} \quad
\hbox{(base group)}.
\end{align*} 
For the reflection system $L=L_{12}^\cT$, the condition
$|H|$ divides $|L|=12$ implies that $H$ could be $1$ or $C_2$.
We have $H_L=1$, and obtain the reflection groups in canonical form
\begin{align*}
G_\cT(L_{12}^\cT,C_2), & \qquad |G|=96,\quad \hbox{$14$ reflections} \quad
\hbox{(higher order group)},  \cr
G_\cT(L_{12}^\cT,1), & \qquad  |G|=48, \quad \hbox{$12$ reflections} \quad
\hbox{(base group)}.
\end{align*}
Since $L_{12}^\cT\subset L_{24}^\cT=\cT$, 
it follows immediately from (\ref{GKLHsinclusion}) that
these four groups satisfy
$$ G_\cT(L_{12}^\cT,1) \subset G_\cT(L_{12}^\cT,C_2) \subset
G_\cT(\cT,Q_8) \subset G_\cT(\cT,\cT). $$
\end{example}

We now make some observations, which apply to our calculations generally.

\begin{itemize}
\item The orders and numbers of reflections are calculated using Lemma \ref{GKLHstructure},
e.g., the reflection group $G=G_\cT(\cT,Q_8)$ has
$$ |G|=2|Q_8||\cT|=384,  \qquad\hbox{and \quad $2|Q_8|+|\cT|-2=38$ reflections}. $$
\item Our method finds all the reflection subgroups of $G_\cT(\cT,\cT)$
	$$ \hbox{{\em which have the canonical form}}. $$
There are also many other reflection subgroups
(not in the canonical form). These include imprimitive reflection groups for 
other choices of $K$, and 
	complex reflection groups
(see Example \ref{noncanonicalsubgroupsExample}).
\item A reflection group can have noncanonical labels, e.g., the base group
for a reflection system $L$, can be indexed by any $H\lhd K$ with $H\subset H_L$.
As an example, $G_\cT(\cT,Q_8)$ has noncanonical labels
$$ G(\cT,\cT,1), \qquad G(\cT,\cT,C_2). $$
\item Whether a reflection subgroup is normal is determined by the 
reflection orbits \cite{W25}. These are the orbits, under conjugation, 
of the subgroups of the reflections for a given root. 
From this theory,
it immediately follows that 
$$ G_\cT(L_{12}^\cT,1) \lhd G_\cT(L_{12}^\cT,C_2), \qquad
G_\cT(\cT,Q_8) \lhd G_\cT(\cT,\cT), $$
and $G_\cT(L_{12}^\cT,C_2)$ is not normal in $G_\cT(\cT,Q_8)$.
\item The number of occurences of a reflection group in canonical form as
a subgroup of $G_K(K,K)$ is given by the size of the equivalence class of
the reflection system, e.g., the reflection groups
$G_\cT(L_{12}^\cT,1)$, $G_\cT(L_{12}^\cT,C_2)$ each appear twelve times
as subgroups of $G_\cT(\cT,\cT)$, as the equivalence class of $L_{12}^\cT$
has size $12$ (Example \ref{TrefsysExample}).
		In particular, $G_K(K,H)\lhd G_K(K,K)$.
\end{itemize}

The base group for the reflection system $L=K$ has the following
general form.

\begin{lemma}
\label{basegroupLeqK}
For the reflection system $L=K$, the base group is given by
$$ H_L = H_K= [K,K]=K^{(1)}=\inpro{\{aba^{-1}b^{-1}:a,b\in K\}} \quad
\hbox{(the commutator subgroup)}. $$
Therefore, the reflection groups in the canonical form for $L=K$ are
	$$ G_K(K,H), \qquad [K,K]\subset H \subset K, $$
where $H$ is a subgroup of $K$ ($[K,K]\subset H$ implies $H$ is normal).
\end{lemma}

\begin{proof} The group generated by the reflections given by $L=K$ contains
$$ \pmat{0&ab\cr(ab)^{-1}&0} \pmat{0&a\cr a^{-1}&0} \pmat{0&b^{-1}\cr b&0}
\pmat{0&1\cr1&0} = \pmat{aba^{-1}b^{-1}&0\cr0&1}, \quad a,b\in K, $$
so that
$$ aba^{-1}b^{-1}\in H_K \Implies [K,K]\subset H_K. $$

	Any subgroup 
	$H$ that contains $[K,K]$
	is normal in $K$, and hence gives a reflection group for $L=K$ 
	in the canonical form. In particular, by taking $H=[K,K]$, we
	conclude that $H_K=[K,K]$.
\end{proof}

\begin{example} 
\label{noncanonicalsubgroupsExample}
The reflection group $G_{\cT}(\cT,\cT)$ has reflection subgroups
	$H$ which are not in the canonical form (for $K=\cT$), including
\begin{align*}
H=\inpro{\pmat{i&0\cr0&1},\pmat{j&0\cr0&1},\pmat{0&i\cr-i&0}}=G_{Q_8}(Q_8,Q_8), 
	&\qquad |H|=128, \quad\hbox{$22$ reflections}, \cr
H=\inpro{\pmat{{1+i+j+k\over2}&0\cr0&1},\pmat{0&i\cr-i&0}}\cong G(6,1,2), 
&\qquad |H|=72, \quad\hbox{$16$ reflections}, \cr
H=\inpro{\pmat{j&0\cr0&1},\pmat{0&1\cr1&0},\pmat{0&i\cr-i&0}}=G_{Q_8}(Q_8,C_4),
&\qquad |H|=64, \quad\hbox{$14$ reflections}.
\end{align*}
\end{example}


\begin{example} (Binary octahedral group) Let $K=\cO$ be the binary octahedral group,
which has normal subgroups $H=1,C_2,Q_8,\cT,\cO$. Elementary calculations,
show that there are five reflection systems of sizes $14$, $18$, $20$, $32$, $48$, 
given by 
\begin{align*}
L_{48}^\cO &:= \{ 1, {1+i\over\sqrt{2}},{1+j\over\sqrt{2}},{1+i+j+k\over2} \},
\quad\hbox{($1$ copy)}\cr
L_{32}^\cO &:= \{ 1,{1+i\over\sqrt{2}}, j, {1+i+j+k\over2} \},
\quad\hbox{($4$ copies)}\cr
L_{20}^\cO &:= \{ 1,{1+i\over\sqrt{2}},{1+i+j+k\over 2},{j-k\over\sqrt{2}} \}, 
\quad\hbox{($10$ copies)}\cr
L_{18}^\cO &:= \{ 1, {1+i\over\sqrt{2}}, {1+i+j+k\over2} \}, 
\quad\hbox{($9$ copies)}\cr
L_{14}^\cO &:= \{1,i,{1+i+j+k\over2},{j-k\over\sqrt{2}} \},
\quad\hbox{($7$ copies)}
\end{align*}
For these reflection systems, the base groups are
\begin{align*}
G_\cO(\cO,\cT), \qquad & |G|=2304, \quad\hbox{$94$ reflections}, \cr
G_\cO(L_{32}^\cO,Q_8), \qquad & |G|=768, \quad\hbox{$46$ reflections}, \cr
G_\cO(L_{20}^\cO,C_2), \qquad & |G|=192, \quad\hbox{$22$ reflections}, \cr
G_\cO(L_{18}^\cO,1), \qquad   & |G|=96, \quad\hbox{$18$ reflections}, \cr
G_\cO(L_{14}^\cO,1), \qquad   & |G|=96, \quad\hbox{$14$ reflections},
\end{align*}
and there is just one higher order group
$$ G_\cO(\cO,\cO), \qquad |G|=4608, \quad\hbox{$142$ reflections}. $$
\end{example}

The inclusions between the reflection systems given above 
(which are mostly obvious from their definitions) are given in Figure \ref{Linclusionsfigure}. We also 
observe that
\begin{equation}
\label{L18L14intersection}
L_{18}^\cO \cap  L_{14}^\cO = L_{12}^\cT, \qquad
L_{14}^\cO = L_{12}^\cT \cup \{ {j-k\over\sqrt{2}},{k-j\over\sqrt{2}}\}.
\end{equation}
In view of the group orders and number of reflections, 
there is one possible isomorphism between the reflection groups for $\cT$ and $\cO$.

\begin{example}
\label{isomorphismI}
By \ref{L18L14intersection}), 
there is an isomorphism $G_\cO(L_{14}^\cO,1)\to G_\cT(L_{12}^\cT,C_2)$ given by
$$ \pmat{0&{j-k\over\sqrt{2}}\cr{k-j\over\sqrt{2}}&0}\mapsto\pmat{-1&0\cr0&1},
\qquad
\pmat{0&b\cr b^{-1}&0}\mapsto \pmat{0&b\cr b^{-1}&0}, \quad 
	b\in L_{12}^\cT=  \{1,i,{1+i+j+k\over2}\}. $$
\end{example}

\begin{example} (Binary icosahedral group) Let $K=\cI$ be the binary icosahedral group,
which has normal subgroups $H=1,C_2,\cI$. Elementary calculations,
show that there are four reflection systems of sizes $20$, $30$, $32$, $120$,
given by
\begin{align*}
L_{120}^\cI &:= \{ 1, i, {1+i+j+k\over2}, {\tau+\gs i-j\over2} \},
\quad\hbox{($1$ copy)}\cr
	L_{32}^\cI &:= \{ 1, {1+i+j+k\over2}, {\tau+\gs i-j\over2},{j-\tau i-\gs k\over2} \}, \quad\hbox{($16$ copies)}\cr
L_{30}^\cI &:= \{ 1, {1+i+j+k\over2}, {\tau+\gs i-j\over2} \}, \quad\hbox{($15$ copies)}\cr
	L_{20}^\cI &:= \{ 1, i, {1+i+j+k\over2},{i+\gs j+\tau k\over2} \}, 
	\quad\hbox{($10$ copies)}.
\end{align*}
For these reflection systems, the base groups are
\begin{align*}
G_\cI(\cI,\cI), \qquad & |G|=28800, \quad\hbox{$358$ reflections}, \cr
G_\cI(L_{32}^\cI,C_2), \qquad & |G|=480, \quad\hbox{$34$ reflections}, \cr
G_\cI(L_{30}^\cI,1), \qquad & |G|=240, \quad\hbox{$30$ reflections}, \cr
G_\cI(L_{20}^\cI,1), \qquad   & |G|=240, \quad\hbox{$20$ reflections},
\end{align*}
and there is one higher order group
$$ G_\cI(L_{20}^\cI,C_2), \qquad   |G|=480, \quad\hbox{$22$ reflections}. $$
\end{example}

There are no further isomorphisms, and so we count $14$ quaternionic reflection groups 
corresponding to the groups $\cT,\cO,\cI$, as in \cite{C80} (which has a few typos).
Taylor \cite{T25} (and in personal correspondence) observes that some imprimitive 
quaternionic reflection groups are conjugate to primitive complex reflection groups, i.e.,
\begin{equation}
\label{primitiveimprimitiveisos}
G_{\cT}(L_{12}^\cT,1)\cong G_{12}, \qquad
G_{\cO}(L_{18}^\cO,1)\cong G_{13}, \qquad
G_{\cI}(L_{30}^\cI,1)\cong G_{22}, 
\end{equation}
and therefore counts $11$ quaternionic reflection groups. 

In Table \ref{TOIgroups-table}, we summarise our classification of the 
reflection groups for $\cT,\cO,\cI$. This includes the reflection orbits for $G$, 
using the notation $n_1 R_{a_1},\ldots, n_m R_{a_m}$ of \cite{W25}, which we 
now explain. Let $R_a$ be the reflection subgroup generated by the reflections 
with root $a$.  
This consists of all reflections with root $a$ and the identity. The reflection type 
is the set of orbits of the reflection subgroups $R_a$ under the conjugation 
action of the reflection group $G$. 
This provides more nuanced information than number of reflections and 
their orders,
and can be used to distinguish reflection groups in which these are equal,
but the groups are not isomorphic. 
The notation $n_j R_{a_j}$ indicates that the orbit of $R_{a_j}$ is of size $n_j$, 
and usually just the abstract type of $R_{a_j}$ is recorded. For our groups
$G_K(L,H)$ we observe that diagonal matrices are conjugated to diagonal matrices
(and similarly for nondiagonal matrices), so that the diagonal reflections form
a single orbit
$$ 2H =  \bigl\{\pmat{H&0\cr0&1},\pmat{1&0\cr0&H}\bigr\}, $$
whilst the order two reflections corresponding to $b\in L$ give
orbits of the form $n C_2$.

For the base group $G=G_K(L,L_H)$, the orbit of a nondiagonal reflection
is given by
$$ \pmat{0&b\cr b^{-1}&0}^G=
\bigl\{ \pmat{0&c\cr c^{-1}}: c\in\Orb(L,b)\bigr\}, \qquad b\in L, $$
where the {\bf orbit} of $b\in L$ is
$$ \Orb(L,b) := \hbox{the closure of $\{b\}$ under 
$x\mapsto a\circ x= ax^{-1}a$, $a\in L$}, $$
and the reflection type is $mC_2$, $m=|\Orb(L,b)|$.
The orbits give a partition of $L$, and if $L$ is generated by $b_1,\ldots,b_m$,
then its orbits are $\Orb(L,b_1),\ldots,\Orb(L,b_m)$ (with some possibly being equal to
each other). The orbit of a nondiagonal matrix under a higher group for $L$ is given
by a union of the orbits for the base group.

\begin{table}[!h]
\caption{The imprimitive reflection groups $G_K(L,H)=\cG(\cL,\cH)$ obtained from the
reflection systems $L$ for $K=\cT,\cO,\cI$.
The base groups have $\cH=\{\}$, and the $\cL$ given is
a generating set for the corresponding reflection system.
The only isomorphism is $G_\cT(L_{12}^\cT,C_2)\cong G_\cO(L_{14}^\cO,1)$,
	and so there are $14$ groups in total. }
	\vskip0.3truecm
\label{TOIgroups-table}
\begin{tabular}{ |  >{$}l<{$} | >{$}c<{$} | >{$}l<{$} | >{$}l<{$} | >{$}l<{$} | >{$}l<{$} | >{$}l<{$} | >{$}l<{$} |}
\hline
	&&&&&&&\\[-0.3cm]
K & |L| & H & \hbox{order} & \hbox{refs} & \hbox{reflection orbits} &
        \cL & \cH \\[0.1cm]
\hline
&&&&&&&\\[-0.3cm]
\cT & 24 & \cT & 1152 & 70 & 2\cT,24C_2 & \{1,i\} & \{{1+i+j+k\over2}\} \\
\cT & 24 & Q_8 & 384 & 38 & 2Q_8,24C_2 & \{1,i,j,{1+i+j+k\over2}\} & \{\} \\
\cT & 12 & C_2 & 96 & 14 & 2C_2,12C_2 & \{1,i,{1+i+j+k\over2}\} & \{-1\} \\
\cT & 12 & 1 & 48 & 12 & 12C_2 & \{1,i,{1+i+j+k\over2}\} & \{\} \\
&&&&&&&\\[-0.3cm]
\cO & 48 & \cO & 4608 & 142 & 2\cO,48C_2 & 
	\{1,{1+i+j+k\over2} \} & \{{1+i\over\sqrt{2}}\} \\
\cO & 48 & \cT & 2304 & 94 & 2\cT,24C_2,24C_2 & 
	\{1, {1+i\over\sqrt{2}},{1+j\over\sqrt{2}}, {1+i+j+k\over2} \} & \{\} \\
\cO & 32 & Q_8 & 768 & 46 & 2Q_8,8C_2,24C_2 & \{1,{1+i\over\sqrt{2}},j,{1+i+j+k\over2}\} & \{\} \\
\cO & 20 & C_2 & 192 & 22 & 2C_2,2C_2,6C_2,12C_2 & \{1,{1+i\over\sqrt{2}},{1+i+j+k\over2},{j-k\over\sqrt{2}}\} & \{\} \\
\cO & 18 & 1 & 96 & 18 & 6C_2,12C_2 & \{1,{1+i\over\sqrt{2}},{1+i+j+k\over2}\} & \{\} \\
\cO & 14 & 1 & 96 & 14 & 2C_2,12C_2 & \{1,i,{1+i+j+k\over2},{j-k\over\sqrt{2}}\} & \{\} \\
&&&&&&&\\[-0.3cm]
\cI & 120 & \cI & 28800 & 358 & 2\cI,120C_2 & \{1,i,{1+i+j+k\over2},{\tau+\gs i-j\over2}\} & \{ \} \\
	\cI & 32 & C_2 & 480 & 34 & 2C_2,2C_2,30C_2 & \{1,{1+i+j+k\over2},{\tau+\gs i-j\over2},{j-\tau i-\gs k\over2}\} & \{ \} \\
\cI & 30 & 1 & 240 & 30 & 30C_2 & \{1,{1+i+j+k\over2},{\tau+\gs i-j\over2}\} & \{ \} \\
\cI & 20 & C_2 & 480 & 22 & 2C_2,20C_2 & \{1,i,{1+i+j+k\over2},{i+\gs j+\tau k\over2}\} & \{-1 \} \\
\cI & 20 & 1 & 240 & 20 & 20C_2 & \{1,i,{1+i+j+k\over2},{i+\gs j+\tau k\over2}\} 
	& \{ \} \\ [0.1cm]
\hline
\end{tabular}
\end{table}


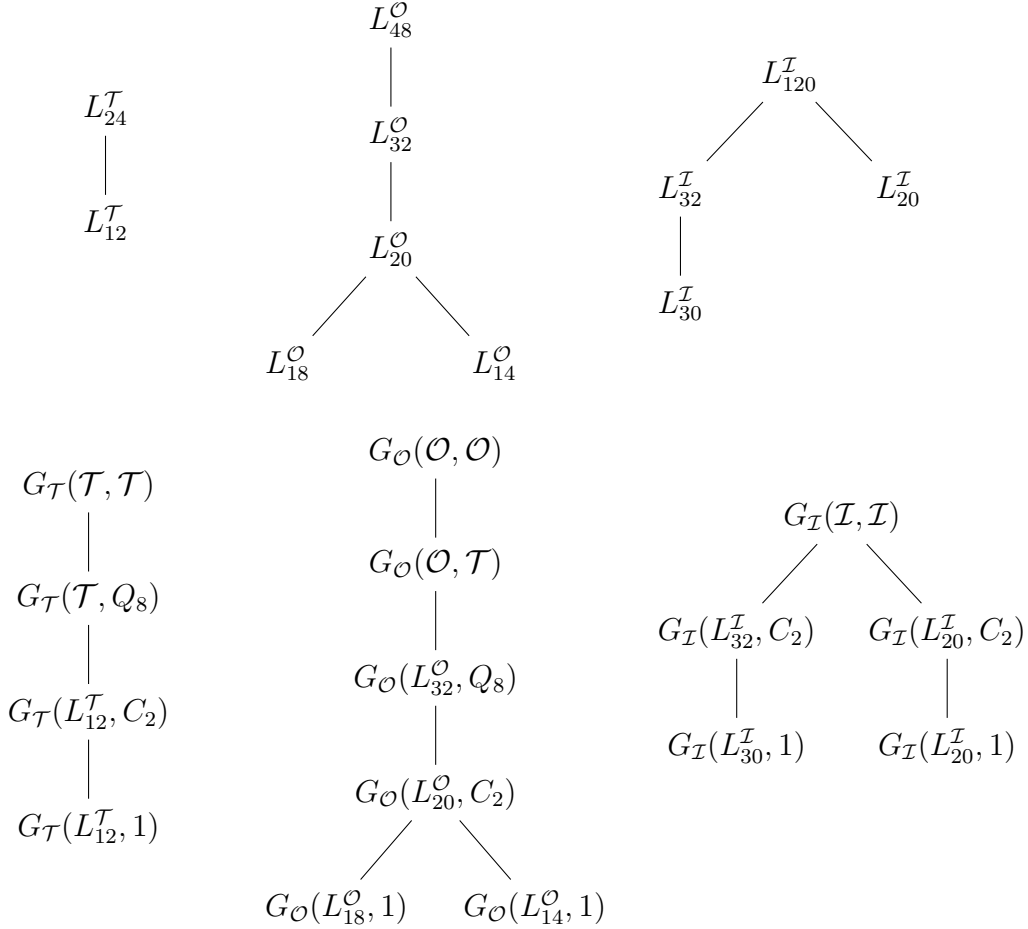
\begin{figure}[!h]
\caption{The inclusions between the reflection systems for $\cT,\cO,\cI$,
	and the inclusions between the reflection groups that correspond to them.}
\label{LTOIinclusionsfigure}
\begin{center}
\begin{tabular}[t]{ p{2.0truecm} p{4.8truecm} p{4.8truecm} }
\begin{tikzpicture}
\matrix (A) [matrix of nodes, row sep=0.8cm, column sep = 0.0 cm]
    {
            $L_{24}^\cT$  \\
            $L_{12}^\cT$  \\
	    $ $ \\
	     \\
    };
    \draw (A-1-1)--(A-2-1);
\end{tikzpicture}
 & 
\begin{tikzpicture}
    \matrix (A) [matrix of nodes, row sep=0.8cm, column sep = 0.5 cm]
    {
            & $L_{48}^\cO$ &  \\
            & $L_{32}^\cO$ &  \\
            & $L_{20}^\cO$ &  \\
            $L_{18}^\cO$ && $L_{14}^\cO$ \\
    };
    \draw (A-1-2)--(A-2-2);
    \draw (A-2-2)--(A-3-2);
    \draw (A-3-2)--(A-4-1);
    \draw (A-3-2)--(A-4-3);
\end{tikzpicture}
	  & 
		\begin{tikzpicture}
    \matrix (A) [matrix of nodes, row sep=0.8cm, column sep = 0.5 cm]
    {
            & $L_{120}^\cI$ &  \\
            $L_{32}^\cI$ && $L_{20}^\cI$  \\
            $L_{30}^\cI$ &&  \\
            \\
    };
    \draw (A-1-2)--(A-2-1);
    \draw (A-1-2)--(A-2-3);
    \draw (A-2-1)--(A-3-1);
\end{tikzpicture}
\end{tabular}
\begin{tabular}[t]{ p{3.0truecm} p{4.8truecm} p{5.8truecm} }
\begin{tikzpicture}
\matrix (A) [matrix of nodes, row sep=0.8cm, column sep = 0.0 cm]
    {
	    $G_\cT(\cT,\cT)$  \\
	    $G_\cT(\cT,Q_8)$  \\
	    $G_\cT(L_{12}^\cT,C_2)$  \\
	    $G_\cT(L_{12}^\cT,1)$  \\
	    $ $ \\
    };
    \draw (A-1-1)--(A-2-1);
    \draw (A-2-1)--(A-3-1);
    \draw (A-3-1)--(A-4-1);
\end{tikzpicture}
 & 
\begin{tikzpicture}
    \matrix (A) [matrix of nodes, row sep=0.8cm, column sep = -1.0 cm]
    {
	    & $G_\cO(\cO,\cO)$ &  \\
	    & $G_\cO(\cO,\cT)$ &  \\
	    & $G_\cO(L_{32}^\cO,Q_8)$ &  \\
	    & $G_\cO(L_{20}^\cO,C_2)$ &  \\
	    $G_\cO(L_{18}^\cO,1)$ && $G_\cO(L_{14}^\cO,1)$ \\
    };
    \draw (A-1-2)--(A-2-2);
    \draw (A-2-2)--(A-3-2);
    \draw (A-3-2)--(A-4-2);
    \draw (A-4-2)--(A-5-1);
    \draw (A-4-2)--(A-5-3);
\end{tikzpicture}
	  & 
		\begin{tikzpicture}
    \matrix (A) [matrix of nodes, row sep=0.8cm, column sep = -0.7 cm]
    {
	    & $G_\cI(\cI,\cI)$ &  \\
			$G_\cI(L_{32}^\cI,C_2)$ && $G_\cI(L_{20}^\cI,C_2) $  \\
			$G_\cI(L_{30}^\cI,1)$ && $G_\cI(L_{20}^\cI,1) $  \\
	    $ $ \\
	    $ $ \\
    };
    \draw (A-1-2)--(A-2-1);
    \draw (A-1-2)--(A-2-3);
    \draw (A-2-1)--(A-3-1);
    \draw (A-2-3)--(A-3-3);
\end{tikzpicture}
\end{tabular}
\end{center}
\end{figure}

\vfil\eject

\section{The dicyclic (binary dihedral) groups}

We now consider the dicyclic (binary dihedral) groups
$$ \cD_n := \inpro{\go,j}, \qquad \go:=\zeta_{2n} = e^{\pi i\over n}, \qquad n\ge 2,
$$ 
where $\go$ is a primitive $2n$-th root of unity.
This group has $4n$ has elements, of two types
\begin{equation}
\label{Dntwotypes}
\go^m, \qquad \go^\ell j=j\go^{-\ell}, \qquad 1\le m,\ell \le 2n.
\end{equation}
We observe that $\go\mapsto\go$, $j\mapsto\go j$ defines an automorphism of $\cD_n$.

The group $\cD_2$ is the quaternion group $Q_8$, which has a slightly special 
structure, since $\go=i$, so that $i$, $j$, $k$ play the same role.
For $n\ge2$, the normal subgroups of $\cD_n$ are
\begin{equation}
\label{DnNormalI}
C_r=\inpro{\go^{2n\over r}}, \quad r\divides 2n, \qquad
\qquad \cD_n,
\end{equation}
and, for $n\ne 2$ even,  
there are two additional nonabelian normal subgroups of order $2n$
\begin{equation}
\label{DnNormalneven}
\cD_{n/2}=\inpro{\go^2,j}, \qquad \inpro{\go^2,\go j}.
\end{equation}
In view of the automorphism $\go\mapsto\go$, $j\mapsto\go j$, which maps the first
group of (\ref{DnNormalneven}) to the second, it suffices to consider just $\cD_{n/2}$.
For $\cD_2=Q_8$, the normal subgroups are those of (\ref{DnNormalI}) and
$\inpro{j}$, $\inpro{k}$ (which are abelian).

To understand what subsets of (\ref{Dntwotypes}) might generate 
reflection systems, we will use Lemma \ref{Ageneratelemma} 
and the following general property.

\begin{lemma} For $x,y\in K$, their closure under $\circ$ contains 
the following elements
\begin{equation}
\label{xyinvformula}
(xy^{-1})^nx \in L(\{x,y\}), \qquad n=0,1,2,\ldots.
\end{equation}
\end{lemma}

\begin{proof} We prove this by strong induction on $n$.
	The $n=0,1$ cases are immediate
$$ x\circ x=xx^{-1}x=x, \qquad x\circ y= xy^{-1}x=(xy^{-1})x. $$
Since $n=2k$ or $n=2k+1$, with $k<n$, we have
\begin{align*}
& (xy^{-1})^kx \circ x = (xy^{-1})^kx x^{-1}(xy^{-1})^kx=(xy^{-1})^{2k}x, \cr
& (xy^{-1})^kx\circ y = (xy^{-1})^kx y^{-1}(xy^{-1})^kx=(xy^{-1})^{2k+1}x,
\end{align*}
which completes the induction. 
\end{proof}

Since the elements of the two types (\ref{Dntwotypes}) are closed under $\circ$, i.e.,
$$ \go^a \circ \go^b = \go^a\go^{-b}\go^a = \go^{2a-b},\qquad
\go^a j \circ \go^b j =  \go^a j (-\go^b j) \go^a j = \go^{2a-b}j, $$
a generating set for a reflection system for $K=\cD_n$ must contain at least one of each.
By taking $x=1$ in (\ref{xyinvformula}), it follows that
$$ L(\{1,\go^{a_1},\ldots,\go^{a_m}\}) $$
is the cyclic group generated by $\go^{a_1},\ldots,\go^{a_m}$, and so we 
can suppose, without loss of generality, 
that a generating set for a reflection system has the form
$$ A= \{1,\go^a,\go^{b_1}j,\ldots, \go^{b_\ell}j\}. $$
By Lemma \ref{Ageneratelemma}, generators for an equivalent reflection system are
$$ A(\go^{b_1}j)^{-1}= \{-\go^{b_1}j,-\go^{a+b_1}j,1,\go^{b_2-b_1},\ldots,\go^{b_\ell-b_1}\}, $$
and so, by the previous reasoning, one may suppose the generating set is
$$ \{-\go^{b_1}j,-\go^{a+b_1}j,1,\go^{b_2-b_1}\} (\go^{b_1}j)
=\{1,\go^a, \go^{b_1}j,\go^{b_2}j\}. $$
We now consider conjugation of the above set by $\go^\gga$. Since
$$ \go^\gga(\go^a)\go^{-\gga}=\go^a, \qquad
\go^\gga(\go^b j)\go^{-\gga} = \go^{2\gga+b}j, $$
we can suppose that the generating set for a reflection system for $K=\cD_n$ has one
of the forms
$$ \{1,\go^a,j,\go^b j\}, \qquad \{1,\go^a,\go j,\go^{b+1} j\}. $$
Since the automorphism $\go\mapsto\go$, $j\mapsto \go j$ of $\cD_n$ maps the
first to the second, it suffices to consider just the first.
The set of indices $(a,b)$ giving the desired reflection systems is
\begin{align}
\label{Andefn}
\gO_n  & := \{(a,b):1\le a\le b\le n,\ a\divides n,\  b\divides n,\ \gcd(a,b)=1\} \cr
&= \{({n\over x},{n\over y}): 1\le y\le x\le n,\ \lcm(x,y)=n\}.
\end{align}
For example,
$$ \gO_6 = \{(1,1),(1,2),(1,3),(1,6),(2,3)\}, \qquad |\gO_6|=5. $$
The size of this set $|\gO_n|$ is the sequence A018892 in on-line encyclopedia of
integer sequences, where several formulas for it are given, including
$$ |\gO_n| = {1\over2}\Bigl(\prod_j (2\ga_j+1)+1\Bigr), \qquad n=\prod_j p_j^{\ga_j}\quad
\hbox{(prime factorisation)}, $$
Thus the size of $\gO_n$ grows with the number of prime factors of $n$ and their
multiplicities.

We now have a key technical result.

\begin{lemma}
The reflection systems for the dicyclic group $\cD_{n}$ are (up to equivalence)
\begin{align}
\label{Labndefn}
L_{(a,b)}^{(n)}  & := L(\{1,\go^a,j,\go^b j\}) \cr
&\ =\{\go^{ma}\}_{1\le m\le {2n\over a}}\cup \{\go^{\ell b}j\}_{1\le \ell\le {2n\over b}}, 
\qquad (a,b)\in \gO_n. 
\end{align}
Each of these has a different number of elements, which is given by
\begin{equation}
\label{Labnsize}
|L_{(a,b)}^{(n)}|={2n\over a}+{2n\over b}.
\end{equation}
\end{lemma}

\begin{proof} We start from our general observation that a reflection system must have the
form
$$ L = L(\{1,\go^a,j,\go^b j\}), $$
for some $a$ and $b$. Since $x^{-1}\in L(\{1,x\})=\inpro{x}$, it makes no difference if
we take $x$ or $x^{-1}$ as a generator, and so we may suppose that 
$1\le a,b\le n$. Further, since we can take any generator of $\inpro{\go^a}$, 
we can suppose that $a\divides n$, and similarly $b\divides n$. Since
\begin{equation}
\label{Labsymmetry}
L_{(a,b)}^{(n)} j^{-1} = L(\{-j,-\go^a j, 1,\go^b \}) 
= L(\{j,\go^{a} j, 1,\go^b \}) = L_{(b,a)}^{(n)},
\end{equation}
so that $L_{(a,b)}^{(n)}$ and $L_{(b,a)}^{(n)}$ are equivalent reflection systems,
we may suppose (arbitrarily) that $a\le b$. 
We now consider the condition on $a$ and $b$ for 
	$\{1,\go^a,j,\go^b j\}$ to generate $\cD_n=\inpro{\go,j}$. The subgroup of $\inpro{\go}$
contained within $\inpro{\go^a,j,\go^b j}$ is that generated by the products of
	$\go^a$ and $\go^b=(\go^b j)j^{-1}$, which is all of $\inpro{\go}$ if and only if
	$\gcd(a,b)=1$. This condition then implies $\cD_n=\inpro{\go^a,j,\go^b j}$,
	and hence we arrive at the index set $\gO_n$. 

Taking $x=\go^a$, $y=1$ and $x=j$, $y=\go^b j$ in (\ref{xyinvformula}) gives
$$\go^{ma}\in L(\{1,\go^a\}, \qquad \go^{\ell b}j \in L(\{j,\go^b j\}), $$
so the elements listed in (\ref{Labndefn}) are in $L=L_{(a,b)}^{(n)}$, and 
contain a generating set for $L$.  They give all of $L$, 
as claimed, since they are closed under $\circ$, i.e., by the calculations
$$ \go^{m_1 a} \circ \go^{m_2 a} = \go^{(2m_1-m_2) a}, \qquad
\go^{\ell_1 b}j \circ \go^{\ell_2 b}j = \go^{(2\ell_1-\ell_2) b}j,  $$
$$ \go^{ma} \circ \go^{\ell b}j = -\go^{\ell b} j = \go^{({n\over b}+\ell)b}, \qquad
	\go^{\ell b}j \circ \go^{ma}  = -\go^{ma} = \go^{({n\over a}+m)a}.	$$
In particular, the size of $L$ is given by (\ref{Labnsize}).

Finally, we show the reflection systems of (\ref{Labndefn})
have different numbers of elements, and consequently are not isomorphic.
This follows from the fact that
$(a,b)\mapsto {1\over a}+{1\over b}$
is bijective on ordered pairs with $a\le b$ and $\gcd(a,b)=1$, and hence on $\gO_n$.
Suppose that
$$  a\le b, \quad a'\le b', \qquad \gcd(a,b)=\gcd(a',b')=1. $$
Then
$$ {1\over a}+{1\over b} = {1\over a'}+{1\over b'}
 \Iff {a+b\over ab} = {a'+b'\over a'b'}, $$
where the latter fractions are in reduced form.
The sum and product uniquely determine a pair $\{a,b\}$, since they
are the roots of the quadratic $(x-a)(x-b)=x^2-(a+b)x+ab$, so that
$\{a,b\}=\{a',b'\}$, and so $(a,b)=(a',b')$.
\end{proof}

The reflection systems for $\cD_n$ satisfy the inclusions implied by their
generators, i.e.,
\begin{equation}
\label{Labn-inclustions-1}
L_{(a',b')}^{(n)} \subset L_{(a,b)}^{(n)},\qquad a\divides a', \quad b\divides b',
\end{equation}
and, by (\ref{Labsymmetry}), also (at the level of equivalence)
\begin{equation}
\label{Labn-inclustions-1}
L_{(a',b')}^{(n)} \subset L_{(a,b)}^{(n)},\qquad a\divides b', \quad b\divides a'.
\end{equation}

\begin{figure}[h]
\begin{center}
	\caption{The reflection systems $L_{(a,b)}^{(n)}$ for $\cD_6$, and their inclusions.}
\begin{tikzpicture}
    \matrix (A) [matrix of nodes, row sep=1.0cm, column sep = -0.5 cm]
    {
	    & $L_{24}=L_{(1,1)}^{(6)}$ & \\
	    $L_{18}=L_{(1,2)}^{(6)}$ && $L_{16}=L_{(1,3)}^{(6)}$ \\
	    $L_{14}=L_{(1,6)}^{(6)}$ && $L_{10}=L_{(2,3)}^{(6)}$ \\
    };
    \draw (A-1-2)--(A-2-1);
    \draw (A-1-2)--(A-2-3);
    \draw (A-2-1)--(A-3-1);
    \draw (A-2-1)--(A-3-3);
    \draw (A-2-3)--(A-3-1);
    \draw (A-2-3)--(A-3-3);
\end{tikzpicture}
\end{center}
\end{figure}
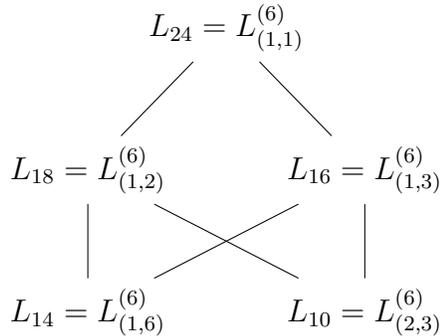

\begin{theorem} 
\label{DnBasegroups}
(Base groups for $\cD_n$)
The base reflection group for the reflection system $L=L_{(a,b)}^{(n)}$,
	$(a,b)\in \gO_n$, of (\ref{Labndefn}) is
\begin{equation}
\label{LabnBasegroup}
G=G_{\cD_n}(L_{(a,b)}^{(n)},C_{n/ab}),
\end{equation}
where $H_L=C_{n/ab}=\inpro{\go^{2ab}}$,  and 
	$$ \hbox{$\displaystyle|G|=8{n^2\over ab}$}, \qquad 
\hbox{$G$ has $\displaystyle {2n\over ab}+{2n\over a}+{2n\over b}-2$ reflections}. $$
Moreover, these groups are all different.
\end{theorem}

\begin{proof}
In view of our general method, it suffices to find $H_L$.
Since $|H_L|$ divides $|L|$, where
$$ |L| = {2n\over a}+{2n\over b}= 2(a+b) {n\over ab}, $$
the cyclic group $\inpro{\go^{2ab}}$
of order ${n\over ab}$,
which is normal in $\cD_n$, could be contained in $H_L$.
This is seen to be the case by the direct calculation:
$$ \pmat{0&\go^{-ma}\cr\go^{ma}}  \pmat{0&j\cr-j&0}
\pmat{0&\go^{\ell b}j\cr(\go^{\ell b}j)^{-1}} \pmat{0&1\cr1&0}
=\pmat{\go^{-ma-\ell b} &0\cr0&\go^{ma-\ell b}}, $$
and choose $m=b$, $\ell=a$ in the above, to obtain
	$$ \pmat{\go^{-2ab} &0\cr0&1}\in G_{\cD_n}(L,H_L) \Implies \go^{2ab}\in H_L
        \Implies \inpro{\go^{2ab}}\subset H_L. $$
By considering all the products of two reflections given by $L$, which are
diagonal matrices, it is easy to conclude that $\inpro{\go^{2ab}}=H_L$,
i.e., $\go^{ab}\not\in H_L$. 

The order and number of reflections follow
from Lemma \ref{GKLHstructure}, by the calculation
$$ |G|=2|\cD_n||H_L|=2(4n){n\over ab}, \qquad
	2|H_L|+|L|-2=2{n\over ab}+\Bigl({2n\over a}+{2n\over b}\Bigr)-2. $$

It follows from the general theory that the base groups for different reflection
systems are different reflection groups.
This can also be seen directly here, by considering the orders and number of reflections.
Suppose that the orders were equal, i.e.,
$$ 8 {n^2\over ab}= 8 {n^2\over a'b'} \Implies {2n\over ab} = {2n\over a'b'}, $$
then from the number of reflections being equal, we have
$$ {2n\over ab}+{2n\over a}+{2n\over b}-2 = {2n\over a'b'}+{2n\over a'}+{2n\over b'}-2
\Implies {1\over a}+{1\over b} = {1\over a'}+{1\over b'}, $$
        and since the map $(a,b)\mapsto {1\over a}+{1\over b}$ is bijective on $\gO_n$,
        we conclude that $(a,b)=(a',b')$.
\end{proof}

Since the order of any normal subgroup $H$ which gives a reflection 
group $G_{\cD_n}(L_{(a,b)}^{(n)},H)$ (in canonical form) must divide
$$ |L_{(a,b)}^{(n)}| = {2n\over a}+{2n\over b}= 2(a+b) {n\over ab}, $$
we can only have $H=\cD_n$, when $(a,b)=(1,1)$, i.e., $L=L_{(1,1)}^{(n)}=\cD_n$, $H_L=\inpro{\go^2}$,
giving the higher order group 
\begin{equation}
\label{DnDnGroup}
G_{\cD_n}(\cD_n,\cD_n) =\cG(\{1,\go,j,\go j\},\{\go,j\}),
\end{equation}
for the base group $G_{\cD_n}(\cD_n,C_n)$. 
For $n$ even, and $(a,b)=(1,1)$, there is 
the second higher order group  
\begin{equation}
\label{DnDn/2Group}
	G_{\cD_n}(\cD_n,\cD_{n/2}) =\cG(\{1,\go,j,\go j\},\{\go^2,j\}),
\end{equation}
for the base group $G_{\cD_n}(\cD_n,C_n)$. 
It is not possible to have $H=\cD_{n/2}$ for any other reflection system
$L_{(a,b)}^{(n)}$, $(a,b)\ne1$, since $|\cD_{n/2}|=2n$ would need to divide
$$ |L_{(a,b)}^{(n)}| = {2n\over a}+{2n\over b}\le {2n\over 1}+{2n\over 2}=3n. $$
These two cases withstanding, there can possibly be only one other higher order group
$$ G_{\cD_n}(L_{(a,b)}^{(n)},C_{2n/ab}). $$

\begin{theorem} 
\label{DnHighergroups}
(Higher order groups for $\cD_n$)
If $ab$ is odd, 
then there is one higher order reflection group for 
the reflection system $L=L_{(a,b)}^{(n)}$,
       $(a,b)\in \gO_n$, of (\ref{Labndefn}) given by 
\begin{equation}
\label{LabnHigergroup}
G=G_{\cD_n}(L_{(a,b)}^{(n)},C_{2n/ab}),
\end{equation}
where $H_L=C_{2n/ab}=\inpro{\go^{2b}}$,  
	and otherwise $G(\cD_n,L_{(a,b)}^{(n)},C_{2n/ab})$ is not in
	the canonical form. For this $G$,
        $$ \hbox{$\displaystyle|G|=16{n^2\over ab}$}, \qquad
\hbox{$G$ has $\displaystyle {4n\over ab}+{2n\over a}+{2n\over b}-2$ reflections}. $$
Moreover, these groups are all different from each other and the base groups.
\end{theorem}

\begin{proof} We first observe that since $\gcd(a,b)=1$, it is not possible to
have $a$ and $b$ both be even. If $a$ is even (so $b$ is odd), then 
	$G(\cD_n,L_{(a,b)}^{(n)},C_{2n/ab})$ contains the reflection
 $$ \pmat{0&\go^{a\over2}\cr \go^{-{a\over2}}&0}
        = \pmat{-1&0\cr0&1} 
        \Bigl[ \pmat{0&1\cr1&0}\pmat{0&\go^a\cr\go^{-a}&0}\Bigr]^{b-1\over2}
       \pmat{0&1\cr1&0}
        \Bigl[ \pmat{0&j\cr-j&0} \pmat{0&\go^b j\cr (\go^b j)^{-1}&0}\Bigr]^{a\over2} , $$
and hence is not in the canonical form. Similarly, for $b$ even, we have
        $$ \pmat{0&\go^{b\over2}j\cr (\go^{b\over2}j)^{-1}&0}
        = \pmat{-1&0\cr0&1}
        \Bigl[ \pmat{0&1\cr1&0}\pmat{0&\go^a\cr\go^{-a}&0}\Bigr]^{b\over2}
        \Bigl[ \pmat{0&j\cr-j&0} \pmat{0&\go^b j\cr (\go^b j)^{-1}&0}\Bigr]^{a-1\over2}
        \pmat{0&j\cr-j&0}, $$
and again $G(\cD_n,L_{(a,b)}^{(n)},C_{2n/ab})$ is not in the canonical form.

For $ab$ is odd, in can be shown that $G(\cD_n,L_{(a,b)}^{(n)},C_{2n/ab})$
is in canonical form, by considering all the diagonal matrices which are 
given by a product of reflections.

Since a higher order group corresponds to a unique reflection system $L$, the 
	higher order groups (if there is one) are not isomorphic to each other, 
	or a base group, which completes the proof.
\end{proof}

\begin{example} 
For the quaternion group $Q_8=\cD_2$, there are two reflection systems
$$ L_{(1,1)}^{(2)}=Q_8=\{1,-1,i,-i,j,-j,k,-k\}, 
\qquad L_{(1,2)}^{(2)}=L(\{1,i,j\})=\{1,-1,i,-i,j,-j\}, $$
of size $8$ and $6$, giving the base groups
\begin{align*}
G_{Q_8}(Q_8,C_2), \qquad & |G|=32, \quad\hbox{$10$ reflections}, \cr
G_{Q_8}(L_{(1,2)}^{(2)},1), \qquad & |G|=16, \quad\hbox{$6$ reflections}, 
\end{align*}
and one higher order group
$$ G_{Q_8}(Q_8,C_4), \qquad   |G|=64, \quad\hbox{$14$ reflections}. $$
The final reflection group is given by (\ref{DnDnGroup}), i.e., 
$$ G_{Q_8}(Q_8,Q_8), \qquad   |G|=128, \quad\hbox{$22$ reflections}. $$
\end{example}

The reflection groups $G_{\cD_n}(L_{(a,b)}^{(n)},C_r)$
of Theorems \ref{DnBasegroups} and \ref{DnHighergroups} are given by the unique
index $[n,a,b,r]\in\gL_n$, where
\begin{equation}
\label{gLndefn}
\gL_n= \bigcup_{(a,b)\in\gO_n}\{ \hbox{$[n,a,b,{n\over ab}]$}\} \cup 
\bigcup_{(a,b)\in\gO_n\atop ab\, {\rm is}\, {\rm odd}} \{\hbox{$[n,a,b,{2n\over ab}]$}\}.
\end{equation}
We will use the notation
\begin{equation}
\label{Gnabrdefn}
G_n(a,b,r)=G(n,a,b,r) := G_{D_n}(L_{(a,b)}^{(n)},C_r)
=\cG(\{1,\go^a,j,\go^b j\},\{\go^{2n\over r}\}),
\end{equation}
for these groups (the $\cH$ above not being required for the base group).
We observe that
\begin{itemize}
\item $abr=n$ for the base group.
\item $abr=2n$ for the higher order group (when $ab$ is odd).
\item $G_n(a,b,r)$ has order $8nr$.
\item $G_n(a,b,r)$ has $2r+{2n\over a}+{2n\over b}-2$ reflections.
\end{itemize}
The size of $\gL_n$ depends on the number of divisors of $2n^2$, which 
we denote $\tau(2n^2)$, i.e.,
$$ |\gL_n|={\tau(2n^2)\over2} +1. $$

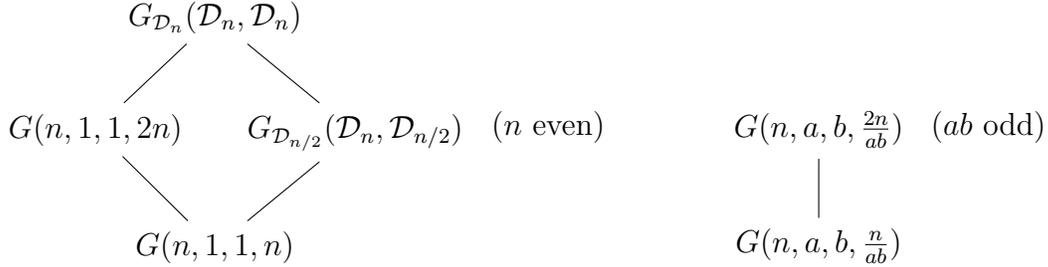
\begin{figure}[h]
\caption{Summary of the reflection groups for $K=\cD_n$, as they evolve from the base group
	$G(n,a,b,{n\over ab})$, $(a,b)\in\gO_n$, for the cases $(a,b)=1$ and $(a,b)\ne1$,
	respectively.}
        \label{Linclusionsfigure}
\begin{center}
\begin{tabular}[t]{ p{6.0truecm} p{2.8truecm} p{2.2truecm} p{2.0truecm} }
\begin{tikzpicture}
\matrix (A) [matrix of nodes, row sep=0.8cm, column sep = -1.0 cm]
    {
	    & $G_{\cD_n}(\cD_n,\cD_n)$ & \\
	    $G(n,1,1,2n)$ && $G_{\cD_{n/2}}(\cD_n,\cD_{n/2})$  \\
	    & $G(n,1,1,n)$ & \\
    };
    \draw (A-1-2)--(A-2-1);
    \draw (A-1-2)--(A-2-3);
    \draw (A-2-1)--(A-3-2);
    \draw (A-2-3)--(A-3-2);
\end{tikzpicture}
 &
	\begin{tikzpicture}
\matrix (A) [matrix of nodes, row sep=0.8cm, column sep = 0.0 cm]
    {
	    $\hbox{($n$ even)}$ \\
                \\
                \\
    };
\end{tikzpicture}
	&
	\begin{tikzpicture}
\matrix (A) [matrix of nodes, row sep=0.8cm, column sep = 0.0 cm]
    {
	     \\
		\\
            $G(n,a,b,{2n\over ab})$ \\
            $G(n,a,b,{n\over ab})$ \\
    };
    \draw (A-3-1)--(A-4-1);
\end{tikzpicture}
	&  \begin{tikzpicture}
\matrix (A) [matrix of nodes, row sep=0.8cm, column sep = 0.0 cm]
    {
	    $\hbox{($ab$ odd)}$ \\
                \\
                \\
    };
\end{tikzpicture}

\end{tabular}
\end{center}
\end{figure}

We now consider when the $G_n(a,b,r)$ for different indices 
can give the same group. Since isomorphic groups have the same reflection orbits,
we first consider these. 
We recall that since $\{1,\go^a,j,\go^b\}$ generates $L=L_{(a,b)}^{(n)}$, we have
$$ L_{(a,b)}^{(n)} = \Orb(L,1)\cup\Orb(L,\go^a) \cup \Orb(L,j)\cup\Orb(L,\go^b j), $$
where
$$ \Orb(L,1)\cup\Orb(L,\go^a) = \{ \go^{ma}\}_{1\le m\le {2n\over a}}, \quad
\Orb(L,j)\cup\Orb(L,\go^b j) = \{ \go^{\ell b}j\}_{1\le\ell\le {2n\over b}}. $$


\begin{lemma} 
\label{LabOrbit-lemma}
($L_{(a,b)}^{(n)}$ orbits)
The orbits of $L=L_{(a,b)}^{(n)}$, $(a,b)\in\gO_n$, satisfy
\begin{enumerate}[(i)]
\item $\Orb(L,\go^a)=\Orb(L,1)$ if and only if ${n\over a}$ is odd.
\item $\Orb(L,\go^b j)=\Orb(L,j)$ if and only if ${n\over b}$ is odd.
\end{enumerate}
In particular, $L$ can have two, three or four orbits.
\end{lemma}

\begin{proof} We first consider the orbit of $\go^a$. 
We observe that $\Orb(L,\go^a)=\Orb(L,1)$ if and only if $1\in\Orb(L,\go^a)$. Since
$$ \go^{ma} \circ \go^a = \go^{(2m-1)a}, \qquad
	\go^{\ell b}j \circ \go^a = -\go^a = \go^{a+n}, $$
the orbit $\Orb(L,\go^a)$ consists of all elements of the form
$\go^{(2m-1)a}$ and $\go^{(2m-1)a+n}$. Such elements can be equal to $1$ if and only if
$(2m-1)a=2n$ or $(2m-1)a=n$, with the latter condition able to be satisfied when
	${n\over a}$ is odd ($m={n+a\over 2a}$), which gives 
$$ \go^{{n+a\over 2a}a} \circ \go^a = \go^{n}=-1 
\Implies 1=j \circ(\go^{{n+a\over 2a}a} \circ \go^a)\in \Orb(L,\go^a). $$

The argument for the orbit of $\go^b j$ follows similarly. Since
$$ \go^{ma} \circ \go^b j = -\go^b j=\go^{b+n}j, \qquad
\go^{\ell b}j \circ \go^b j = \go^{(2\ell-1)b}j, $$
the elements of $\Orb(L,\go^b j)$ have the form
$\go^{(2\ell-1)b}j$ or $\go^{(2\ell-1)b+n}j$, and this can be equal to $j$ 
only in the second case, when ${b\over n}$ is odd and $2\ell-1={n\over b}$, 
which gives
$$ \go^{{n+b\over2b} b}j \circ \go^b j
        = \go^{n}j=-j 
	\Implies j = 1 \circ (\go^{{n+b\over2b} b}j \circ \go^b j) \in\Orb(L,\go^b j). $$
It is easy enough to choose $n,a,b$ so that there are all possibilities of 
	equality and inequality in (i) and (ii), e.g., $L$ has exactly two orbits 
	if and only if $n$ is odd.
\end{proof}

From the proof of Lemma \ref{LabOrbit-lemma}, it follows that
$\Orb(L,\go^a)$ and $\Orb(L,1)$ have the same size (whether or not they are equal),
and similarly for the other pair of orbits, so we conclude 
\begin{equation}
\label{Orbwa}
|\Orb(L,\go^a)|=|\Orb(L,1)|=
\begin{cases}
        {2n\over a}, & \hbox{${n\over a}$ is odd}; \cr
        {n\over a}, & \hbox{${n\over a}$ is even},
\end{cases}
\end{equation}
\begin{equation}
\label{Orbwbj}
|\Orb(L,\go^b j)|=|\Orb(L,j)|=
\begin{cases}
        {2n\over b}, & \hbox{${n\over b}$ is odd}; \cr
        {n\over b}, & \hbox{${n\over b}$ is even},
\end{cases}
\end{equation}

\begin{theorem} 
\label{abapbp-lemma}
(Isomorphisms) 
The only reflection groups $G=G(n,a,b,r)$ for different indices which are isomorphic are given by the 
infinite family of index pairs
\begin{equation}
\label{infinitefamilyiso}
[n,1,n,2], \quad [2n,2,n,1], \qquad \hbox{$n$ odd},
\end{equation}
with an isomorphism $G(n,1,n,2)\to G(2n,2,n,1)$ given by
$$ \pmat{1&0\cr0&-1}\mapsto \pmat{0&k\cr-k&0}, \qquad
	\pmat{0&b\cr b^{-1}&0}\mapsto \pmat{0&b\cr b^{-1}&0}, \quad b\in\{1,\go,j\}, \quad
	\go:=e^{\pi i\over n}. $$
\end{theorem}

\begin{proof}
Suppose that $G=G(n,a,b,r)$ and $G'=G(n',a',b',r')$ are isomorphic, then all 
their reflections have order $2$, and by Lemma \ref{isomorphismlemma}, we have
$$ H=C_2, \quad H'=1 \qquad |G'|=2|G|\Implies n'=2n, $$
so that $r=2$, and $r'=1$, i.e., ${n'\over a'b'}=1$. There are two possibilities, 
depending on whether $G$ is a higher order group, or a base group, respectively, i.e.,
\begin{equation}
\label{quad-one}
{n\over ab}=1, \quad {n'\over a'b'}=1 \Implies a'b'=n'=2n=2ab,
\end{equation}
\begin{equation}
\label{quad-two}
{n\over ab}=2, \quad {n'\over a'b'}=1 \Implies a'b'=n'=2n=4ab.
\end{equation}
From $|L'|=|L|+2$, these give
\begin{equation}
\label{lin-one}
2n'{a'+b'\over a'b'} = 2n{a+b\over ab}+2 \Implies a'+b'=a+b+1,
\end{equation}
\begin{equation}
\label{lin-two}
2n'{a'+b'\over a'b'} = 2n{a+b\over ab}+2 \Implies a'+b'=2(a+b)+1.
\end{equation}
Since $a'b'=2n$, either $a'$ or $b'$ is even (they have no common factors), with
	the other being odd. In the first case, (\ref{lin-one}) implies $a+b=a'+b'-1$ is even. 
But $a$ and $b$ have no common factors, so they must both be odd, and hence $n=ab$ is odd.

The reflection system $L_{(a',b')}^{(2n)}$ for $G(2n,a',b',1)$ must have an orbit of size $2$.
Since ${2n\over b'}<{2n\over a'}$, Lemma \ref{LabOrbit-lemma} and (\ref{Orbwbj})
give two possibilities
$$ {4n\over b'}=2 \quad \hbox{(${2n\over b'}$ odd)}, \qquad
{2n\over b'}=2 \quad \hbox{(${2n\over b'}$ even)}, $$
i.e., 
$$ a'=1, \quad b'=2n, \qquad a'=2,\quad b'=n. $$

Consider the case: $a'=1$, $b'=2n$. From 
(\ref{quad-one}), (\ref{lin-one}) and (\ref{quad-two}), (\ref{lin-two}),
we have
$$ a+b=2n, \quad ab=n \Implies a = n-\sqrt{n^2-n}, $$
	$$ a+b=n, \quad ab={n\over2} \Implies a = {n-\sqrt{n^2-2n}\over2}. $$
Since $n$ and $n-1$ have no common factors, the square root 
of $n^2-n=n(n-1)$ above cannot be an integer, 
and similarly $n^2-2n=n(n-2)$ cannot be a perfect square for $n\ne2$.
Thus an isomorphism with $a'=1$, $b'=2n$ is not possible.

For the other case: $a'=2$, $b'=n$, 
(\ref{quad-one}), (\ref{lin-one}) and (\ref{quad-two}), (\ref{lin-two}), give
$$ a+b=n+1, \quad ab=n \Implies a = 1, \quad b=n, $$
$$ a+b={n+1\over2}, \quad ab={n\over2} \Implies a = {n+1-\sqrt{n^2-6n+1}\over4}. $$
	The latter case is not possible, since for $n^2-6n+1$, $n\ne0$, to be a perfect square,
	we must have $n=6$, which gives $a={7-1\over4}={3\over2}$.
For the first case,
%
we obtain the indices $[n,1,n,2]$, $[2n,2,n,1]$, where $n$ is odd (as previously observed),
	and corresponding groups
	$$ G=G(n,1,n,2), \qquad G'=G(2n,2,n,1). $$
	Let $\go$ be a primitive $4n=2(2n)$ root of unity, so the reflection systems are
		$$ L= L_{(n,1)}^{(n)} =L(\{1,\go^2,j,(\go^2)^n j\}) = L(\{1,\go^2,j,-j\}) = L(\{1,\go^2,j\}), $$
		$$ L'=L_{(2,n)}^{(2n)}=L(\{1,\go^2,j,\go^n j\})=L(\{1,\go^2,j,ij\}) = L(\{1,\go^2,j,k\}), $$
	where $|L'|=|L|+2$, gives 
		$$ L' = L \cup \{k,-k\}. $$
	Both reflection groups have three orbits of reflections (of order $2$), respectively
	$$\bigl\{ \pmat{1&0\cr0&-1},\pmat{-1&0\cr0&1}\bigr\}, \quad
	\bigl\{\pmat{0&j\cr-j&0},\pmat{0&-j\cr j&0}\bigr\}, \quad
	\bigl\{\pmat{0&b\cr b^{-1}&0} : b \in L\setminus\{j,-j\}\bigr\}, $$
	$$ \bigl\{\pmat{0&k\cr-k&0},\pmat{0&-k\cr k&0}\bigr\}, \quad
	\bigl\{\pmat{0&j\cr-j&0},\pmat{0&-j\cr j&0}\bigr\}, \quad
	\bigl\{\pmat{0&b\cr b^{-1}&0} : b \in L\setminus\{j,-j\}\bigr\}. $$
	It is easy to verify that
	$$ \pmat{1&0\cr0&-1}\mapsto \pmat{0&k\cr-k&0}, \qquad
		\pmat{0&b\cr b^{-1}&0}\mapsto \pmat{0&b\cr b^{-1}&0}, \quad b\in L=L_{(1,n)}^{(n)}, $$
		gives an isomorphism $G\to G'$ (which is defined by its action on generators).
	%
\end{proof}

\begin{table}[h]
\caption{The imprimitve reflection groups for the dicyclic group $\cD_n$, $n\ge2$, including 
the particular case $\cD_2=Q_8$. Those for which $H$ is cyclic are of the form $G_n(a,b,r)$.
The base groups have $\cH=\{\}$, and the $\cL$ given is
a generating set for the corresponding reflection system.
The only isomorphism is $G_{\cD_n}(L_{(1,n)}^{(n)},C_2)\cong G_{\cD_{2n}}(L_{(2,n)}^{(2n)},1)$,
	$n$ odd.
	}
\label{Dn-refs-table}
\vskip0.25truecm
\begin{tabular}{ |  >{$}l<{$} | >{$}c<{$} | >{$}l<{$} | >{$}l<{$} | >{$}l<{$} | >{$}l<{$} | >{$}l<{$} | >{$}l<{$} |}
\hline
        &&&&&&&\\[-0.3cm]
	K & L & |L| & H & \hbox{order} & \hbox{reflections} &
        \cL & \cH \\[0.1cm]
\hline
&&&&&&&\\[-0.3cm]
Q_8 & Q_8 & 8 & Q_8 & 128 & 22 & \{1,i\} & \{{i,j}\} \\
Q_8 & Q_8 & 8 & C_4 & 64 & 14 & \{1,i\} & \{j\} \\
Q_8 & Q_8 & 8 & C_2 & 32 & 10 & \{1,i,j,k\} & \{\} \\
	Q_8 & L_{(1,2)}^{(2)} & 6 & 1 & 16 & 6 & \{1,i,j\} & \{\} \\
&&&&&&&\\[-0.3cm]
\cD_n & \cD_n & 4n & \cD_n & 32n^2 & 12n-2 & \{1,\go,j,\go j\} & \{\go,j\} \\
\cD_n & \cD_n & 4n & \cD_{n/2}=\inpro{\go^2,j}  & 16n^2 & 8n-2 & \{1,\go,j,\go j\} & \{\go^2,j\} \\
\cD_n & L_{(a,b)}^{(n)} & {2n\over a}+{2n\over b} & C_{2n\over ab}=\inpro{\go^{ab}} & 16{n^2\over ab} & {4n\over ab}+{2n\over a}+{2n\over b}-2 & \{1,\go^a,j,\go^b j \} & \{\go^{ab}\} \\
\cD_n & L_{(a,b)}^{(n)} & {2n\over a}+{2n\over b} & C_{n\over ab}=\inpro{\go^{2ab}} & 8{n^2\over ab} & {2n\over ab}+{2n\over a}+{2n\over b}-2 & \{1,\go^a,j,\go^b j \} & \{\} \\ [0.2cm]
\hline
\end{tabular}
\end{table}

The stipulation of the orbit size being $2$ is necessary in the above proof.

\begin{corollary}
\label{iso-coroll}
Reflection groups $G(n,a,b,r)$ and $G(n',a',b',r')$, $n<n'$, 
with reflections of order only two,
have the same order and the same number of reflections, but are not isomorphic,
in the following cases
\begin{enumerate}[(i)]
\item $n=ab$ is odd, $a\ne1$, and 
$$ (a+b+1)^2-8ab=c^2, \qquad c\in\{1,3,5,\ldots\}, $$
with the groups being $G(n,a,b,2)$ and $G(2n,{a+b+1+c\over2},{a+b+1-c\over2},1)$.
\item $n=2ab$ is even, and 
	$$ (2(a+b)+1)^2-8ab=c^2, \qquad c\in\{1,3,5,\ldots\}, $$
		with the groups being $G(n,a,b,2)$ and $G(2n,{2(a+b)+1+c\over2},{2(a+b)+1-c\over2},1)$.

\end{enumerate}
\end{corollary} 

\begin{proof} The assertion that the groups have reflections of order only two, 
the same order and number of 
reflections, with $r=2$, $r'=1$, in the proof of Theorem \ref{abapbp-lemma}, leads to the 
necessary and sufficient conditions 
(\ref{quad-one}), (\ref{lin-one}) and (\ref{quad-two}), (\ref{lin-two}).

For the first pair,  $n=ab$ is odd, $n'=2n$, and
$$ a'b'=2ab, \quad a'+b'=a+b+1. $$
The condition $a\ne1$ excludes the isomorphism of Theorem \ref{abapbp-lemma}.
Using $a'+b'=a+b+1$ to eliminate $a'$ (equivalently $b'$) from $a'b'=2ab$ gives the
quadratic equation
$$ x^2-(a+b+1)x+2ab=0, $$
which has $a'$ and $b'$ as roots. The descriminant of this equation is
$(a+b+1)^2-8ab$, which therefore must be a square, say $c^2$, giving the formulas
$$ a' = {a+b+1-c\over 2}, \qquad b' = {a+b+1+c\over 2}. $$
Since $c=b'-a'$, of integers with different parity, 
we conclude that $c$ must be odd.

The argument for the second pair is similar, with $a'$ and $b'$ being the
roots of 
$$ x^2-(2(a+b)+1)x+4ab=0, $$
i.e., 
	$$ a' = {2(a+b)+1-c\over 2}, \qquad b' = {2(a+b)+1+c\over 2}. $$
where 
$$ (2(a+b)+1)^2-16ab=c^2. $$
\end{proof}

\begin{example}
There appear to be infinitely many type (i)  index pairs of  
Corollary \ref{iso-coroll},
with the first few being
\begin{align*}
[ 273, 7, 39, 2 ],\ [ 546, 21, 26, 1 ], \qquad
& [ 315, 7, 45, 2 ], \ [ 630, 18, 35, 1 ], \cr
[ 357, 7, 51, 2 ], \ [ 714, 17, 42, 1 ], \qquad
& [ 975, 13, 75, 2 ],\ [ 1950, 39, 50, 1 ], \cr
[ 1001, 11, 91, 2 ],\ [ 2002, 26, 77, 1 ], \qquad
& [ 1105, 13, 85, 2 ],\ [ 2210, 34, 65, 1 ], \cr
[ 1365, 15, 91, 2 ],\ [ 2730, 42, 65, 1 ], \qquad
& [ 1885, 13, 145, 2 ],\ [ 3770, 29, 130, 1 ].
\end{align*}
There are infinitely many of type (ii), including the family
$$ [2m(2m-1),m,2m-1,2], \quad [4m(2m-1),2m-1,2m,1], \qquad m=1,2,3,\ldots, $$
which may in fact be all of them.
\end{example}

We now consider how our reflection groups for $\cD_n$ relate those of Cohen \cite{C80}.
No isomorphisms were found in \cite{C80}, since the Lemma 2.3 (stated without proof there) says that
one must have $K=K'$ to obtain isomorphic groups. This is false 
(our Theorem \ref{abapbp-lemma}, and also Example \ref{isomorphismI}).

Table I of \cite{C80} has seven lines for the reflection 
groups for $\cD_n$. The lines $4,5$ give the reflection groups
$G_{\cD_n}(\cD_n,\cD_{n/2})$ and $G_{\cD_n}(\cD_n,\cD_{n})$, respectively,
and the lines $1,2,3,6,7$ give groups of the form 
$G_{\cD_n}(L_{(a,b)}^{(n)},C_r)=G(n,a,b,r)$.
The parameters for these groups are summarised in Table \ref{Cohen'sLabgroups}
(see Example \ref{indicesforCohenstable}). 
These do not account for all the groups $G(n,a,b,r)$ in 
our classification, with the first reflection groups missing being given by the indices
\begin{equation}
\label{firstnewgroups}
[6,1,3,4], \quad [ 9, 1, 3, 6 ], \quad [ 10, 1, 5, 4 ], \quad
[ 12, 1, 3, 8 ], \quad [ 15, 1, 5, 6 ], \quad [ 15, 1, 3, 10 ], 
\end{equation}
$$ [ 18, 1, 3, 12 ], \quad [ 18, 1, 9, 4 ], \quad [ 20, 1, 5, 8 ], \quad
[ 21, 1, 7, 6 ], \quad [ 21, 1, 3, 14 ], \quad [ 22, 1, 11, 4 ], $$
$$ [ 24, 1, 3, 16 ], \quad [ 25, 1, 5, 10 ], \quad [ 26, 1, 13, 4 ], \quad
[ 27, 1, 9, 6 ], \quad [ 27, 1, 3, 18 ], \quad [ 28, 1, 7, 8 ], $$
$$ [ 30, 1, 3, 20 ], \quad [ 30, 3, 5, 4 ], \quad [ 30, 1, 15, 4 ], \quad
[ 30, 1, 5, 12 ],\quad [ 33, 1, 3, 22 ],\quad [ 33, 1, 11, 6 ], \quad \ldots . $$
These are the indices $[n,a,b,r]\in\gL_n$, with
$$ (a,b)\ne(1,1), \qquad r\ne 2, \qquad r\nmid n, \qquad r\divides 2n,
 $$
and give higher order groups, with the base group appearing in either 
line 2 or 3. 

\begin{center}
	\begin{table}[h]
        \fontsize{10pt}{10pt}\selectfont
	\caption{Cohen's Table I (lines $1,2,3,6,7$) }
\vskip0.3truecm
\label{Cohen'sLabgroups}
        \begin{tabular}{ |  >{$}l<{$} | >{$}c<{$} | >{$\scriptstyle}p{2.6truecm}<{$} | >{$\scriptstyle}p{3.8truecm}<{$} | >{$}l<{$} | >{$}l<{$} | }
\hline
&&&&&\\[-0.1cm]
                K & H & \hbox{$|L|$} & \hbox{$\ga\in{\rm Aut}(K/H)$} &  |G| & [n,a,b,r]^\dagger
         \\[0.15cm]
\hline
&&&&&\\[0.0cm]
                \cD_{m} & C_{2m}  & \hbox{$4m$} & \hbox{$1$} & 16m^2 & {\scriptstyle [m,1,1,2m] }  \\
&&&&&\\[-0.1cm]
                \cD_{2m\ell} & C_{2m} &  2m\{\gcd(2\ell,r+1) & \hbox{$\ga_r,\ $}   {0\le r\le \ell,} \ r\ {\rm odd}  & 32m^2\ell & {\scriptstyle [2m\ell,\gcd(\ell,{r-1\over2}),\gcd(\ell,{r+1\over2}),2m] }  \\
                & & +\gcd(2\ell,r-1) \} & \ell=\gcd(\ell,{r+1\over2})\gcd(\ell,{r-1\over2}) &&  \\
&&&&&\\[-0.1cm]
                \cD_{(2m+1)\ell} & C_{2m+1} & (2m+1)\{\gcd(2\ell,r-1) & \hbox{$\gb_r,\ $}   {0\le r\le \ell,} \ r\ {\rm odd} & 8(2m+1)^2\ell & {\scriptstyle [(2m+1)\ell,\gcd(\ell,{r-1\over2}),\gcd(\ell,{r+1\over2}),2m+1]} \\
                & & +\gcd(2\ell,r+1) \} & \ell=\gcd(\ell,{r+1\over2})\gcd(\ell,{r-1\over2}) &&  \\
&&&&&\\[-0.1cm]
                \cD_{2m+1} & C_2 & 2\{\gcd(2m+1,r+1) & \hbox{$\ga_r,\ $}   {0\le r\le m,} \qquad 2m+1= & 16(2m+1) & {\scriptstyle [2m+1,\gcd(2m+1,r-1),\gcd(2m+1,r+1),2] }  \\
                 & & {+\gcd(2m+1,r-1) \}}  & {\gcd(2m+1,r+1)\gcd(2m+1,r-1)} & &  \\
&&&&&\\[-0.1cm]

                \cD_m & 1 & \gcd(2m,r+1) & \hbox{$\gb_r,\ $}   {0\le r\le m,} \ r\ {\rm odd}  & 8m & {\scriptstyle [m,\gcd(m,{r-1\over2}),\gcd(m,{r+1\over2}),1] }  \\
                 & & {+\gcd(2m,r-1) }  & {m=\gcd(m,{r+1\over2})\gcd(m,{r-1\over2})} & &  \\
&&&&&\\[-0.3cm]
&&&&&\\
\hline
\multicolumn{4}{l}{\hbox{  }} \\
\multicolumn{6}{l}{\hbox{\footnotesize $\dagger$ The index given may not have $a\le b$
		for all choices of $m,\ell,r$. Lines $2,3,5$ are the base group ($abr=n$).}}
\end{tabular}
\end{table}
        \end{center}

We consider the first of the new reflection groups given by the indices
(\ref{firstnewgroups}).

\begin{example}
\label{groups192}
(Groups of order $192$) The previously unknown
reflection group with index $[6, 1, 3, 4]$ has order $192$. 
The collection of all reflection groups of this order is

\begin{center}
\begin{tabular}{ |  >{$}l<{$} | >{$}l<{$} | >{$}l<{$} | >{$}l<{$} | >{$}l<{$} | }
\hline
&&&&\\[-0.3cm]
	G & \hbox{\rm identifier} & \hbox{\rm refs} & |L| & H  \\[0.1cm]
\hline
&&&&\\[-0.3cm]
[ 6, 1, 3, 4 ] & \inpro{192, 385} & 22 & 16 & C_4 \cr
[ 12, 1, 6, 2 ] & \inpro{192, 1312} & 30 & 28 & C_2 \cr
[ 12, 2, 3, 2 ] & \inpro{192, 1330} & 22 & 20 & C_2 \cr
[ 24, 1, 24, 1 ] & \inpro{192, 463} & 50 & 50 & 1 \cr
[ 24, 3, 8, 1 ] & \inpro{192, 471} & 22 & 22 & 1 \cr 
G_\cO(L_{20}^\cO,C_2) & \inpro{192, 1486} & 22 & 20 & C_2 \cr 
G_9 & \inpro{192, 963} & 30 & & \cr
\hline
\end{tabular}
\end{center}
We observe that there are four groups with $22$ reflections, each for a different
	group $K$, including an example of type (ii) in Corollary \ref{iso-coroll} 
	(reflections of order two only, $c=5$).
\end{example}

We now consider the possibility of an isomorphism between a reflection group 
$G$ given by a polyhedral group (see Table \ref{TOIgroups-table}) and a reflection group $G'$ given by a dicyclic group
(see Table \ref{Dn-refs-table}). 
For such an isomorphism 
\begin{itemize}
\item The $H$ for the group $G$ must be a normal subgroup of
some dicyclic group, i.e., $H=1,C_2,Q_8$ 
($\cT,\cO,\cI$ are not subgroups of a dicyclic group).
\item The group $G$ must have at least 
two reflection orbits (Lemma \ref{LabOrbit-lemma}).
\end{itemize}
In view of Table \ref{TOIgroups-table}, this narrows down the possible orders of 
$G$ to $48,96,192,384,480,768$.
Since there are finitely many reflection groups of those orders, we can simply 
examine the reflection structure of each group, or their isomorphism class
(as in Example \ref{groups192} for groups of order $192$). This calculation,
see Example \ref{groups48-96-}, gives the following.

\begin{proposition} There are no isomorphisms between the reflection groups given by 
the polyhedral groups and those given by dicyclic groups.
\end{proposition}

\begin{example}
\label{groups48-96-}
The reflection groups of orders $48,96,384,480,768$ are the following
\begin{center}
\vskip-0.4truecm
\begin{tabular}{ p{7.5truecm} p{0truecm} p{8truecm} }
\begin{tabular}{ |  >{$}l<{$} | >{$}l<{$} | >{$}l<{$} | >{$}l<{$} | >{$}l<{$} | }
\hline
&&&&\\[-0.3cm]
        G & \hbox{\rm identifier} & \hbox{\rm refs} & |L| & H  \\[0.1cm]
\hline
&&&&\\[-0.3cm]
[ 3, 1, 3, 2 ]* & \inpro{48, 39} & 10 & 8 & C_2 \cr
[ 6, 2, 3, 1 ]* & \inpro{48, 39} & 10 & 10 & 1 \cr
[ 6, 1, 6, 1 ] & \inpro{48, 37} & 14 & 14 & 1 \cr
G_\cT(L_{12}^\cT,1) & \inpro{48, 29} & 12 & 12 & 1 \cr
G_{12} & \inpro{48, 29} & 12 & & \cr
G_6 & \inpro{48, 33} & 14 & & \cr
&&&&\\[-0.3cm]
[ 6, 1, 3, 2 ] & \inpro{96, 217} & 18 & 16 & C_2 \cr
[ 12, 3, 4, 1 ] & \inpro{96, 119} & 14 & 14 & 1 \cr
[ 12, 1, 12, 1 ] & \inpro{96, 111} & 26 & 26 & 1 \cr
G_\cT(L_{12}^\cT,C_2)* & \inpro{96, 190} & 14 & 12 & C_2 \cr
G_\cO(L_{14}^\cO,1)* & \inpro{96, 190} & 14 & 14 & 1 \cr
G_\cO(L_{18}^\cO,1) & \inpro{96, 192} & 18 & 18 & 1 \cr
G_{13} & \inpro{96, 192} & 18 & & \\
G_{8} & \inpro{96, 67} & 18 & & \cr
\hline
\end{tabular}
		&&
\begin{tabular}{ |  >{$}l<{$} | >{$}l<{$} | >{$}l<{$} | >{$}l<{$} | >{$}l<{$} | }
\hline
&&&&\\[-0.3cm]
        G & \hbox{\rm identifier} & \hbox{\rm refs} & |L| & H  \\[0.1cm]
\hline
&&&&\\[-0.3cm]
[ 12, 1, 3, 4 ] & \inpro{384, 12471} & 38 & 32 & C_4 \cr
[ 24, 1, 12, 2 ] & \inpro{384, 14591} & 54 & 52 & C_2 \cr
[ 24, 3, 4, 2 ] & \inpro{384, 14609} & 30 & 28 & C_2 \cr
[ 48, 1, 48, 1 ] & \inpro{384, 1945} & 98 & 98 & 1 \cr
[ 48, 3, 16, 1 ] & \inpro{384, 1952} & 38 & 38 & 1 \cr
G_\cT(\cT,Q_8) & \inpro{384, 18130} & 38 & 24 & Q_8 \cr
&&&&\\[-0.3cm]
[ 30, 1, 15, 2 ] & \inpro{480, 1177} & 66 & 64 & C_2 \cr
[ 30, 3, 5, 2 ] & \inpro{480, 1077} & 34 & 32 & C_2 \cr
[ 60, 3, 20, 1 ] & \inpro{480, 349} & 46 & 46 & 1 \cr
[ 60, 5, 12, 1 ] & \inpro{480, 346} & 34 & 34 & 1 \cr
[ 60, 4, 15, 1 ] & \inpro{480, 877} & 38 & 38 & 1 \cr
[ 60, 1, 60, 1 ] & \inpro{480, 869} & 122 & 122 & 1 \cr
G_\cI(L_{32}^\cI,C_2) & \inpro{480, 957} & 34 & 32 & C_2 \cr
G_\cI(L_{20}^\cI,C_2) & \inpro{480, 953} & 22 & 20 & C_2 \cr
\hline
\end{tabular}
\end{tabular}
\end{center}
Here, 
we include the primitive Shephard-Todd groups, with $*$ indicating
an isomorphism.
\end{example}

\section{Computations and concluding remarks}

Our results combine to give the following classification.

\begin{theorem} 
\label{CombinedClassification}
(Classification) Every finite imprimitive irreducible rank two quaternionic 
reflection group can be written uniquely in the canonical form, except for the following
\begin{equation}
\label{allisomorphisms}
G_\cO(L_{14}^\cO,1)\cong G_\cT(L_{12}^\cT,C_2), \qquad
G(n,1,n,2)\cong G(2n,2,n,1), \quad \hbox{$n$ odd},
\end{equation}
which have two canonical forms. 
Tables \ref{TOIgroups-table} and \ref{Dn-refs-table} list all the possible canonical forms.
\end{theorem}

There is an informal group (including Taylor, Bellamy, Schmitt, Thiel) working on a systemic 
labelling and implementation of the quaternionic reflection groups in computer algebra packages 
such as {\tt magma} and {\tt gap}. The classification of Theorem \ref{CombinedClassification}
suggests a unique label, except for the cases (\ref{allisomorphisms}), where possible options are
\begin{itemize}
\item Choose the higher order group, so the group is defined over a smaller $K$.
\item Choose the base group, so that every reflection system leads to a reflection
	group in the classification.
\item Live with the isomorphisms (\ref{allisomorphisms}).
\end{itemize}

The $\cL$ and $\cH$ listed in Tables \ref{TOIgroups-table} and \ref{Dn-refs-table},
easily allow for the 
groups $G_K(L,H)$, reflection systems $L(\cL)$ and orbits $\Orb(L,a)$ to be 
calculated. For example, in {\tt magma}, the groups $G(n,a,b,r)$ are given by 
\begin{verbatim}
Gn := function(n,a,b,r)
  F:=CyclotomicField(4*n); z:=RootOfUnity(2*n); Z:=IntegerRing();
  Q<i,j,k>:=QuaternionAlgebra<F|-1,-1>;
  w:=Q!((z+ComplexConjugate(z))/2+(z-ComplexConjugate(z))/(2*RootOfUnity(4))*i);
  gensH:={Matrix(Q,2,[w^(Z!(2*n/r)),0,0,1])};
  gensL:={Matrix(Q,2,[0,c,c^-1,0]) : c in {1,w^a,j,w^b*j}};
  return MatrixGroup< 2,Q | gensH join gensL >;
end function;
\end{verbatim}
with the index set $\gL_n$ being given by
\begin{verbatim}
Ln := function(n)
  inds:={}; Z:=IntegerRing();
  for x in [1..n] do for y in [x.. n] do
    if LCM(x,y) eq n then
      a:=Z!(n/y); b:=Z!(n/x);
      Include(~inds,[n,a,b,Z!(n/(a*b))]);
      if IsOdd(a*b) then
        Include(~inds,[n,a,b,Z!(2*n/(a*b))]);
      end if;
    end if;
  end for; end for;
  return inds;
end function;
\end{verbatim}
and $L(\cL)$ can be recursively calculated from $\cL$ (similarly for $\Orb(L,a)$) via
\begin{verbatim}
function GenerateL(elts)
  eltsiterate:=elts join {a*b^-1*a: a in elts, b in elts};
  if eltsiterate eq elts
    then return elts;
    else return GenerateL(eltsiterate);
  end if;
end function;
\end{verbatim}

Finally, we 
consider the classification of the imprimitive quaternionic 
reflection groups $G\subset U(\HH^n)$ of rank $n$ greater than two, in the context of our methods.

Just as we chose $1\in L$, so that
$$ \hbox{$\pmat{0&1\cr1&0}$ is a reflection in $G$}, $$
leading 
to the canonical form (\ref{GKLHelements}), $G$ may be put in a canonical form where its
elements have the form $B P_\gs$, where $B$ is an $n\times n$ diagonal matrix with entries in $K$, 
and $P_\gs$ a permutation matrix given by $\gs\in S_n$. Here $P_\gs$ for a transposition 
$\gs=(\ga\, \gb)$ is a reflection which fixes the orthogonal complement $V_{\ga,\gb}^\perp$ of 
$V_{\ga,\gb}=\spam_\HH\{e_\ga,e_\gb\}$, and the vector $e_\ga+e_\gb$. The (parabolic) subgroup of $G$
generated by the reflections in $G$ which fix $V_{\ga,\gb}^\perp$ pointwise, 
is a rank two imprimitive reflection group. 
This leads to the following (Theorem 2.9 of \cite{C80}).

\begin{theorem} 
(Classification)
Every finite imprimitive irreducible quaternionic reflection group $G$ of rank $n\ge3$ can be written
uniquely in the canonical form
\begin{equation}
\label{GnKHdefn}		
G_n(K,H) 
:= \Bigl\{ \pmat{b_1&&&\cr &\ddots&&\cr &&b_{n-1} &\cr &&& (b_1\cdots b_{n-1})^{-1}h} P_\gs 
:\ b_1,\ldots,b_{n-1}\in K,\ h \in H,\  \gs\in S_n \Bigl\},
\end{equation}
where $K$ is a finite subgroup of $U(\HH)$, and $H$ is
	a subgroup with $[K,K]\subset H\subset K$.

In particular, these groups have
$$ |G_n(K,H)| = n!|H||K|^{n-1}, \qquad \hbox{$G_n(K,H)$ has $n(|H|-1)+|K|$ reflections}. $$
\end{theorem}

\begin{proof}
The key idea is that because $G$ contains the subgroup $\{P_\gs\}_{\gs\in S_n}\cong S_n$, 
it acts the same on any $d$-dimensional subspace of the form $\spam_{\HH}\{e_{\ga_1},\ldots,e_{\ga_d}\}$.
	If its restriction to a two-dimensional subspace is given by (\ref{GKLHelements}), then 
	on a three-dimensional subspace, we have
$$ \pmat{b&0&0\cr 0&b^{-1}&0\cr 0&0&1}
= \pmat{1&0&0\cr 0&0&1\cr 0&1&0} \pmat{b&0&0\cr 0&b_\ga&0\cr 0&0&1}
\pmat{0&0&1\cr0&1&0\cr 1&0&0} \pmat{b&0&0\cr 0&b_\ga&0\cr 0&0&1}^{-1}
\pmat{0&1&0\cr 0&0&1\cr 1&0&0}, \quad b\in K, $$
so that on a two-dimensional subspace it acts as an imprimitive reflection group 
with $L=K$, and hence by Lemma \ref{basegroupLeqK}, $[K,K]\subset H\subset K$.
Again, in view of (\ref{GKLHelements}), $G$ must contain the diagonal elements
$$ \pmat{b_1&&&&\cr &b_1^{-1}&&&\cr &&1&&\cr &&&\ddots&\cr &&&&1}
\pmat{1&&&&\cr &b_1b_2&&&\cr &&(b_1b_2)^{-1}&&\cr &&&\ddots&\cr &&&&1}
\cdots
\pmat{1&&&&\cr &\ddots&&&\cr &&1&&\cr &&&b_1\cdots b_{n-1}&\cr &&&&(b_1\cdots b_{n-1})^{-1}h},
$$
and hence the elements of (\ref{GnKHdefn}).

It remains to show that the set of matrices in (\ref{GnKHdefn}) forms a group,
as observed (without proof) in \cite{C80}.
It is sufficient to show that a product of two diagonal matrices from (\ref{GnKHdefn}) 
is another, i.e.,
$$ (a_1\cdots a_{n-1})^{-1}h_a (b_1\cdots b_{n-1})^{-1}h_b
= (a_1b_1\cdots a_{n-1}b_{n-1})^{-1}h, \qquad\exists h\in H, $$
or, equivalently, 
\begin{equation}
\label{conditiontoprove}
(a_1\cdots a_{n-1})^{-1}(b_1\cdots b_{n-1})^{-1}H
= (a_1b_1\cdots a_{n-1}b_{n-1})^{-1}H. 
\end{equation}
We can prove this by induction on $n$, using the fact that 
$[K,K]\subset H$ implies the cosets are permutation invariant, i.e.,
$a_1a_2\cdots a_n H= a_{\gs1}a_{\gs2}\cdots a_{\gs n}H$, for any permutation $\gs$.
	Suppose that $(\ref{conditiontoprove})$ holds for $n\ge2$ (it holds for 
	$n=2$ by Lemma \ref{GKLHstructure}). Then, we calculate
\begin{align*}
(a_1\cdots a_{n})^{-1}(b_1\cdots b_{n})^{-1}H
& = a_n^{-1} (a_1\cdots a_{n-1})^{-1} b_n^{-1}(b_1\cdots b_{n-1})^{-1}H \cr
& = b_n^{-1} a_n^{-1} (a_1\cdots a_{n-1})^{-1} (b_1\cdots b_{n-1})^{-1}H \quad\hbox{(permutation invariance)}\cr
	& = (a_n b_n)^{-1}  (a_1b_1\cdots a_{n-1}b_{n-1})^{-1}H \quad\hbox{(inductive hypothesis)} \cr
& = (a_1b_1\cdots a_{n}b_{n})^{-1}H,
\end{align*}
which completes the induction.
\end{proof}

We observe, from Tables \ref{TOIgroups-table} and \ref{Dn-refs-table}, 
that $G_2(K,H)=G_K(K,H)$, and for $n\ge3$ either
\begin{itemize}
\item There is one imprimitive reflection group $G_n(K,K)$, when $K=\cI,\cD_n$ ($n$ is even).
\item There are two imprimitive reflection groups $G_n(K,H)$ (base group) and $G_n(K,K)$, 
	when $K=\cT,\cO,\cD_n$ ($n$ is odd), where $H=Q_8,\cT,C_n$ (respectively).
\end{itemize}

\section{Acknowledgment}

I wish to thank Don Taylor for sharing his extensive knowledge of
reflection groups and their implementation in magma, by answering my 
numerous questions over many years.

\section{Appendix}

We recall Stringham's list of the elements of the binary polyhedral groups \cite{St1881}

\center{\it {The Double-Tetrahedron Group}
$$ i^\gep, \qquad j^\gep, \qquad k^\gep, $$
$$ \Bigl({1+i+j+k\over2}\Bigr)^\eta, \quad
\Bigl({1-i-j+k\over2}\Bigr)^\eta, \quad
\Bigl({1+i-j-k\over2}\Bigr)^\eta, \quad
\Bigl({1-i+j-k\over2}\Bigr)^\eta, $$
$$ \gep=1,2,3,4; \qquad \eta=1,2,3,4,5,6; \qquad N=24. $$
	\center{The Double-Oktahedron Group}
$$ \Bigl({1+i\over\sqrt{2}}\Bigr)^\gep, \qquad \Bigl({1+j\over\sqrt{2}}\Bigr)^\gep, 
	\qquad \Bigl({1+k\over\sqrt{2}}\Bigr)^\gep, $$
$$ \Bigl({1+i+j+k\over2}\Bigr)^\eta, \quad
\Bigl({1-i-j+k\over2}\Bigr)^\eta, \quad
\Bigl({1+i-j-k\over2}\Bigr)^\eta, \quad
\Bigl({1-i+j-k\over2}\Bigr)^\eta, $$
$$ \Bigl({j+k\over\sqrt{2}}\Bigr)^\zeta, \qquad \Bigl({k+i\over\sqrt{2}}\Bigr)^\zeta, 
	\qquad \Bigl({i+j\over\sqrt{2}}\Bigr)^\zeta, $$
$$ \Bigl({j-k\over\sqrt{2}}\Bigr)^\zeta, \qquad \Bigl({k-i\over\sqrt{2}}\Bigr)^\zeta, 
	\qquad \Bigl({i-j\over\sqrt{2}}\Bigr)^\zeta, $$
$$ \gep=1,2,3,4,5,6,7,8; \qquad \eta=1,2,3,4,5,6; \qquad
	\zeta=1,2,3,4, \qquad N=48. $$
	\center{The Double-Ikosahedron Group}
$$ i^\gep, \qquad j^\gep, \qquad k^\gep, $$
	$$ \Bigl({i+\gs j+\tau k\over2}\Bigr)^\gep, \qquad \Bigl({j+\gs k+\tau i\over2}\Bigr)^\gep, \qquad \Bigl({k+\gs i+\tau j\over2}\Bigr)^\gep, $$
	$$ \Bigl({i-\gs j+\tau k\over2}\Bigr)^\gep, \qquad \Bigl({j-\gs k-\tau i\over2}\Bigr)^\gep, \qquad \Bigl({k+\gs i-\tau j\over2}\Bigr)^\gep, $$
	$$ \Bigl({i+\gs j-\tau k\over2}\Bigr)^\gep, \qquad \Bigl({j-\gs k+\tau i\over2}\Bigr)^\gep, \qquad \Bigl({k-\gs i-\tau j\over2}\Bigr)^\gep, $$
	$$ \Bigl({i-\gs j-\tau k\over2}\Bigr)^\gep, \qquad \Bigl({j+\gs k-\tau i\over2}\Bigr)^\gep, \qquad \Bigl({k-\gs i+\tau j\over2}\Bigr)^\gep, $$
	$$ \Bigl({1+i+j+k\over2}\Bigr)^\eta, \quad
\Bigl({1-i-j+k\over2}\Bigr)^\eta, \quad
\Bigl({1+i-j-k\over2}\Bigr)^\eta, \quad
\Bigl({1-i+j-k\over2}\Bigr)^\eta, $$
$$ \Bigl({1+\tau j+\gs k\over2}\Bigr)^\eta, \qquad \Bigl({1+\tau k+\gs i\over2}\Bigr)^\eta, \qquad \Bigl({1+\tau i+\gs j\over2}\Bigr)^\eta, $$
$$ \Bigl({1+\tau j-\gs k\over2}\Bigr)^\eta, \qquad \Bigl({1+\tau k-\gs i\over2}\Bigr)^\eta, \qquad \Bigl({1+\tau i-\gs j\over2}\Bigr)^\eta, $$
$$ \Bigl({\tau+\gs j+ k\over2}\Bigr)^\zeta, \qquad \Bigl({\tau+\gs k+ i\over2}\Bigr)^\zeta, \qquad \Bigl({\tau+\gs i+ j\over2}\Bigr)^\zeta, $$
$$ \Bigl({\tau+\gs j-k\over2}\Bigr)^\zeta, \qquad \Bigl({\tau+\gs k-i\over2}\Bigr)^\zeta, \qquad \Bigl({\tau+\gs i- j\over2}\Bigr)^\zeta, $$
$$ \gep=1,2,3,4,; \qquad \eta=1,2,3,4,5,6; \qquad
	\zeta=1,2,\ldots,10, \qquad N=120. $$
	$$ \tau={1+\sqrt{5}\over2}, \qquad \gs={1-\sqrt{5}\over2} $$}

\begin{example} 
\label{indicesforCohenstable}
We give an illustration of how we deduced the indices 
for the reflection groups $G=G_{\cD_n}(L,H)$ in Table I of \cite{C80}.
The first entry has
$$ n=m, \quad H=C_{2m}, \quad |L|=4m, \quad |G|=16m^2. $$
Since 
$$ |L| =4m = {2n\over 1}+{2n\over 1}, $$
we must have $(a,b)=(1,1)$, and the index is
$$ [n,a,b,r]=[m,1,1,2m].  $$ 
The base group with index $[m,1,1,m]$ appears in lines $2$ and $3$
	for $n=m$ even or odd. 

The second entry has
$$ n=2m\ell, \quad H=C_{2m}, \quad |G|=32m^2\ell, $$
and
$$ |L| = 2m\bigl(\gcd(2\ell,r-1)+\gcd(2\ell,r+1)\bigr), 
\quad 0\le r\le\ell,\quad\hbox{$r$ odd}, $$
where
$$ \ell = \gcd(\ell,{r+1\over2})\gcd(\ell,{r-1\over2}). $$
By writing
$$ |L| = {2n\over a}+{2n\over b}
= {2(2m\ell)\over2\ell }\bigl(\gcd(2\ell,r-1)+\gcd(2\ell,r+1)\bigr), $$
we deduce that
$$ a ={2\ell\over\gcd(2\ell,r+1)}=\gcd(\ell,{r-1\over2}),
\qquad 
b ={2\ell\over\gcd(2\ell,r-1)}=\gcd(\ell,{r+1\over2}). $$
Here, one can have $a\le b$ or $a\ge b$.
Since $n=2m\ell$, we have $a\divides n$, $b\divides n$, and we verify
\begin{align*}
\gcd(a,b) &= \gcd(\gcd(\ell,{r-1\over2}),\gcd(\ell,{r+1\over2})) \cr
& = {1\over2}\gcd(\gcd(2\ell,r-1),\gcd(2\ell,r+1)) = {1\over2} 2 =1,
\end{align*}
which follows from 
$\gcd(r-1,r+1)=2$, for $r$ odd. 
Hence the index for this group is
$$ [n,a,b,r] = [2m\ell,\gcd(\ell,{r-1\over2}),\gcd(\ell,{r+1\over2}),2m], $$
and in particular, we obtain the index $[n,1,1,n]$ for $n$ even, 
	of the base group for line $1$.
\end{example}

\bibliographystyle{alpha}
\bibliography{references}

\begin{thebibliography}{BST23}

\bibitem[Bli17]{B17}
Hans~Frederik Blichfeldt.
\newblock {\em Finite collineation groups : with an introduction to the theory
  of operators and substitution groups.}
\newblock University of Chicago science series. University of Chicago Press,
  Chicago, 1917.

\bibitem[BST23]{BST23}
Gwyn Bellamy, Johannes Schmitt, and Ulrich Thiel.
\newblock On parabolic subgroups of symplectic reflection groups.
\newblock {\em Glasg. Math. J.}, 65(2):401--413, 2023.

\bibitem[BW25]{BW25}
Zachary Buckley and Shayne Waldron.
\newblock Quaternionic {MUBs} in {$\HH^2$} and their reflection symmetries.
\newblock preprint, 9 2025.

\bibitem[Coh80]{C80}
Arjeh~M. Cohen.
\newblock Finite quaternionic reflection groups.
\newblock {\em J. Algebra}, 64(2):293--324, 1980.

\bibitem[Cox34]{C34}
H.~S.~M. Coxeter.
\newblock Discrete groups generated by reflections.
\newblock {\em Ann. of Math. (2)}, 35(3):588--621, 1934.

\bibitem[CS03]{CS03}
John~H. Conway and Derek~A. Smith.
\newblock {\em On quaternions and octonions: their geometry, arithmetic, and
  symmetry}.
\newblock A K Peters, Ltd., Natick, MA, 2003.

\bibitem[DZ24]{DZ24}
Anirudh Deb and Gabi Zafrir.
\newblock {$\cal{N} = 5$} {SCFT}s and quaternionic reflection groups.
\newblock {\em J. High Energy Phys.}, (8):Paper No. 17, 58, 2024.

\bibitem[Kan81]{K81}
William~M. Kantor.
\newblock Generation of linear groups.
\newblock In {\em The geometric vein}, pages 497--509. Springer, New
  York-Berlin, 1981.

\bibitem[LT09]{LT09}
Gustav~I. Lehrer and Donald~E. Taylor.
\newblock {\em Unitary reflection groups}, volume~20 of {\em Australian
  Mathematical Society Lecture Series}.
\newblock Cambridge University Press, Cambridge, 2009.

\bibitem[Sch23]{S23}
Johannes Schmitt.
\newblock {\em On Q-factorial terminalizations of symplectic linear quotient
  singularities}.
\newblock doctoralthesis, Rheinland-Pf{\"a}lzische Technische Universit{\"a}t
  Kaiserslautern-Landau, 2023.

\bibitem[ST54]{ST54}
G.~C. Shephard and J.~A. Todd.
\newblock Finite unitary reflection groups.
\newblock {\em Canad. J. Math.}, 6:274--304, 1954.

\bibitem[Str81]{St1881}
W.~I. Stringham.
\newblock Determination of the {F}inite {Q}uaternion {G}roups.
\newblock {\em Amer. J. Math.}, 4(1-4):345--357, 1881.

\bibitem[Tay25]{T25}
Don Taylor.
\newblock Finite quaternionic reflection groups in magma.
\newblock talk, 4 2025.

\bibitem[Voi21]{V21}
John Voight.
\newblock {\em Quaternion algebras}, volume 288 of {\em Graduate Texts in
  Mathematics}.
\newblock Springer, Cham, [2021] \copyright 2021.

\bibitem[Wal24]{W24}
Shayne Waldron.
\newblock The geometry of the six quaternionic equiangular lines in
  $\mathbb{H}^2$.
\newblock {\em arXiv}, 2411.16766, 2024.

\bibitem[Wal25]{W25}
Shayne Waldron.
\newblock The lattice of normal reflection subgroups of an irreducible
  reflection group.
\newblock preprint, 6 2025.

\bibitem[Zha97]{Z97}
Fuzhen Zhang.
\newblock Quaternions and matrices of quaternions.
\newblock {\em Linear Algebra Appl.}, 251:21--57, 1997.

\end{thebibliography}
\nocite{*}

\end{document}